\documentclass[11pt]{article}

\usepackage[latin1]{inputenc}
\usepackage{amsmath,amssymb,amsthm,graphicx}
\usepackage[dvips]{epsfig}
\usepackage{amsfonts}
\usepackage{enumerate}
\usepackage{soul}
\usepackage{xcolor}
\usepackage{cancel}

\newtheorem{theorem}{Theorem}[section]
\newtheorem{proposition}[theorem]{Proposition}
\newtheorem{corollary}[theorem]{Corollary}

\newtheorem{lemma}[theorem]{Lemma}
\newtheorem{claim}[theorem]{Claim}

\theoremstyle{remark}

\theoremstyle{definition}

\newcommand{\Z}{\mathbb{Z}}

\newcommand{\A}{\mathcal{A}}

\newcommand{\Nz}{\mathbb{N}_0}
\newcommand{\N}{\mathbb{N}}

\newcommand{\As}{\mathbf{A}}

\newcommand{\ver}{\mathcal{V}}
\newcommand{\pinch}{\mathcal{V}^{\bullet}}

\newcommand{\aleq}{\leq^\circ}
\newcommand{\ageq}{\geq^\circ}

\newcommand{\lsh}{$(L)$}
\newcommand{\ssh}{$(S)$}

\renewcommand{\aa}{\mathbf{a}}
\newcommand{\bb}{\mathbf{b}}

\newcommand{\jj}{\mathbf{j}}
\newcommand{\pp}{\mathbf{p}}
\newcommand{\qq}{\mathbf{q}}
\renewcommand{\ss}{\mathbf{s}}
\newcommand{\ttt}{\mathbf{t}}
\newcommand{\uu}{\mathbf{u}}
\newcommand{\vv}{\mathbf{v}}
\newcommand{\xx}{\mathbf{x}}
\newcommand{\yy}{\mathbf{y}}
\newcommand{\zz}{\mathbf{z}}
\newcommand{\OO}{\mathbf{0}}

\newcommand{\Aall}{A^{all}}

\renewcommand{\k}{\Bbbk}

\newcommand{\heading}[1]{\medskip\par\noindent{\bf #1}}

\DeclareMathOperator{\reg}{reg}

\DeclareMathOperator{\rk}{rk}

\newcommand{\covers}{{\ > \kern-4mm \cdot \kern+2mm\ }}
\newcommand{\covered}{{\ < \kern-2mm \cdot\ }}

\newcommand{\hull}[1]{\langle #1 \rangle}

\title{Shellability of the higher pinched Veronese posets}

\author{
Martin Tancer\thanks{
  Institutionen f\"{o}r matematik, Kungliga Tekniska H\"{o}gskolan, 100~44
  Stockholm. Supported by the G\"{o}ran Gustafsson postdoctoral fellowship.
}
}

\begin{document}
\date{}

\maketitle


\begin{abstract}
 The pinched Veronese poset $\pinch_n$ is the poset with ground set consisting
 of all non-negative integer vectors of length $n$ such that the sum of their
 coordinates is divisible by $n$ with exception of the vector $(1,\dots,1)$. 
 For two
 vectors $\aa$ and $\bb$ in $\pinch_n$ we have $\aa \preceq \bb$ if and only if
 $\bb - \aa$ belongs to the ground set of $\pinch_n$.
 We show that every interval in $\pinch_n$ is shellable for $n \geq 4$.  

 In order to obtain the result, we develop a new method for showing that a
 poset is shellable. This method differs from classical lexicographic
 shellability.

 Shellability of intervals in $\pinch_n$ has consequences in
 commutative algebra. As a corollary we obtain a combinatorial proof of 
 the fact that the pinched Veronese ring is Koszul for $n \geq 4$.
 (This also follows from a result by Conca, Herzog, Trung and Valla.)
\end{abstract}

\section{Introduction}
In this paper we focus on the following question: Is every interval in the
pinched Veronese poset shellable? (Cohen-Macaulay?) Let us explain this
question and its background in detail.

By the $m$-th \emph{Veronese poset with on $n$ generators}, denoted as
$(\ver_{m,n}, \leq)$, we mean the following poset. 
Its ground set consists of non-negative integer vectors of length $n$ 
such the sum of their coordinates is divisible by $m$. The partial order on
$\ver_{m,n}$ is given so that $\aa \leq \bb$ if and only if $\aa$ is less or
equal to $\bb$ in each coordinate. It is not hard to see that every interval in 
$\ver_{m,n}$ is shellable and therefore Cohen-Macaulay.

If we set $m = n$, we just speak of the $n$-th \emph{Veronese poset}
$\ver_{n} := \ver_{n,n}$. We can pinch this poset in the following way. We remove
the distinguished vector $\jj$ which contains $1$ in each coordinate. We also
remove order relations between vectors that differ exactly by $\jj$ (making
them incomparable). In this way we thus obtain the $n$-th \emph{pinched Veronese
poset} $(\pinch_n,\preceq)$; see Figure~\ref{f:621}. (More details on this poset are discussed in
Section~\ref{s:Veronese}.)  It is very interesting that removing this single
element $\jj$ (and the corresponding order relations) strongly influences
understanding the properties of the poset. 

\begin{figure}
\begin{center}
\includegraphics{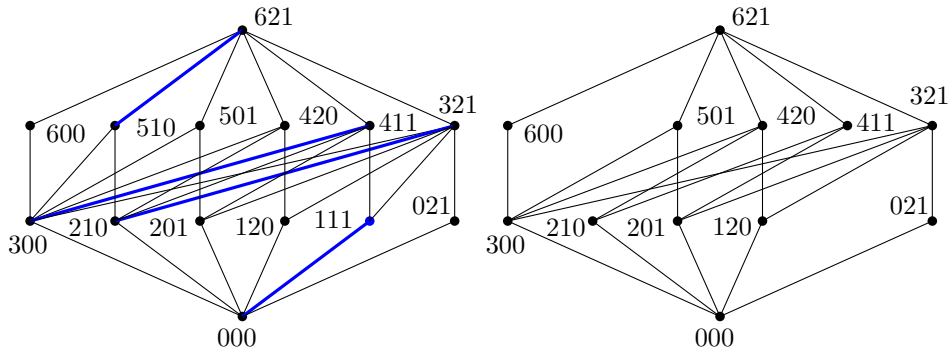}
\caption{An example of an interval in $\ver_3$ and $\pinch_3$. The edges that
have to be removed from $\ver_3$ in order to obtain $\pinch_3$ are emphasized
on the left.}
\label{f:621}
\end{center}
\end{figure}

On the algebraic side, it follows that the $n$-th pinched Veronese ring is
Koszul for $n \geq 4$ from a result by Conca, Herzog, Trung and
Valla~\cite{concaherzogtrungvalla97} (we will discuss this in more detail
below). This is equivalent to stating that every
interval in $\pinch_n$ is Cohen-Macaulay; see~\cite[Corollary
2.2]{peeva-reiner-sturmfels98}. Later on, Caviglia~\cite{caviglia09} showed that the
third pinched Veronese ring is Koszul. 
The methods used in~\cite{caviglia09} are based on computer calculations. Recently, a more general result was found by Caviglia and Conca~\cite{caviglia-conca13}
without the use of computer. 


Our task is to focus on the combinatorial side of this question. That is, we
focus on shellability of intervals in the pinched Veronese poset remarking that shellability
implies Cohen-Macaulayness. 
We also remark that Cohen-Macaulayness of a poset implies several deep intrinsic
properties of the poset. For example certain enumerative properties. The reader
is referred, for example, to~\cite{bjorner-garsia-stanley82} for more details
on Cohen-Macaulayness.

We develop a new method for showing that a certain
poset is shellable. Using this method, we are able to prove the following
theorem.

\begin{theorem}
\label{t:vn}
Let $n \geq 4$. For any $\zz \in \pinch_n$ the interval $[\OO,\zz]$ in
$\pinch_n$ is a shellable poset, where $\OO$ is the zero vector of length $n$.
\end{theorem}

Note that we do not lose anything by considering intervals $[\OO, \zz]$ only,
since an interval $[\aa,\bb]$ is isomorphic to $[\OO, \bb - \aa]$.


Our motivation for proving Theorem~\ref{t:vn} can be seen from two sides. On
one hand, the pinched Veronese poset is an interesting poset from a
combinatorial point of view and it is interesting to understand its
combinatorial properties. Especially, if its combinatorial properties have
further consequences in commutative algebra (see the text at the end of this
section).

On the other hand, Theorem~\ref{t:vn} can be seen as a testing example for a
new method for showing that a certain poset is shellable. We establish
inductive criteria showing that a certain poset $P$ is shellable assuming that
several subposets of $P$ are shellable and that $P$ satisfy few other
properties. Let us remark that, in general, our method differs from a very
standard tool which is lexicographic shellability.

A small drawback of our method is that it requires quite 
technical case analysis checking that all inductive criteria are satisfied. In
this part, the main message for the reader is that the analysis can be done
(still, it is fully included in the paper).

\heading{The third pinched Veronese poset.} 
The reader might wonder what is the importance of our assumption $n \geq 4$ in
Theorem~\ref{t:vn}. The case $n=1$ does not make sense. The case $n= 2$ makes the most
sense (in relation to the algebraic side of the question) if the elements
$(\alpha_1,\alpha_2)$ are further removed from the poset whenever $\alpha_1$
and $\alpha_2$ are odd. However, in this case $\pinch_2$ is isomorphic to
$\ver_{1,2}$.

The only real issue occurs when $n = 3$. In this case, our method, as stated in
section~\ref{s:crit}, does not suffice to prove shellability of $\pinch_3$. In
fact, it is possible to show that some intervals in $\pinch_3$ are not
lexicographically shellable. It turns out that the reason why some intervals in
$\pinch_3$ are not lexicographically shellable also implies limitations for our
method. Maybe a further improvement of our method might
yield a solution for $n = 3$.

\heading{More detailed relation to commutative algebra.}
Let us fix an integer $n$ and consider a subset $\A$ of $\N_0^n$. For
simplicity we assume that the sum of the coordinates of all vectors in $\A$
equals a fixed integer $m$. Given a commutative field $\k$ we consider the
ring $\k[\A]$ as a subring of $\k[x_1, \dots, x_n]$ generated by all monomials
$x^{\aa}$ for $\aa \in \A$ where $x^{\aa} = x_1^{a_1}\cdots x_n^{a_n}$ if $\aa
= (a_1, \dots, a_n)$.

We can also associate a poset $P(\A)$ to $\A$ in the following way. We let
$\Lambda$ to consist of those vectors in $\N_0^n$ that are non-negative
integer combinations of vectors from $\A$ (including zero). 
Then we set $P(\A) = (\Lambda, \leq_{\A})$ where $\aa \leq_{\A} \bb$ if and
only if $\bb - \aa \in \Lambda$.

Cohen-Macaulayness of intervals in $P(\A)$ is related to the Koszul property
of $\k[\A]$ in the following way.

\begin{proposition}[{\cite[Corollary 2.2]{peeva-reiner-sturmfels98}}]
  \label{p:prs}
  The ring $\k[\A]$ is Koszul if and only if every interval in $P(\A)$ is
  Cohen-Macaulay over $\k$. 
\end{proposition}

The reader is referred, for example, to~\cite{froberg99} for more information
about the importance of the Koszul property. 

If we set $\A_{m,n}$ to consist of all vectors in $\N^n_0$ whose coordinates
sum to $m$ we get $P(\A_{m,n}) = \ver_{m,n}$. Similarly, if we set
$\A^\bullet_n$ to $\A_{n,n} \setminus \{\jj\}$, we get $P(\A^\bullet_n) =
\pinch_n$. Thus, we have the following corollary of Theorem~\ref{t:vn} and
Proposition~\ref{p:prs}. 

\begin{corollary}
\label{c:koszul}
  The ring $\k[\A^\bullet_n]$ is Koszul for any $n \geq 4$.
\end{corollary}

As we mentioned above, Corollary~\ref{c:koszul} also follows from the result of
Conca et al~\cite{concaherzogtrungvalla97}, thus our contribution for algebraic
side is a combinatorial proof of this corollary.

For completeness, we explain how to derive Corollary~\ref{c:koszul} from
Corollary~6.10~(2) in~\cite{concaherzogtrungvalla97}. We set $I$ to be the
ideal $(x_1^2, \dots, x_n^2)$ in $\k[x_1,\dots,x_n]$. It is generated by a regular
sequence since $x_i^2$ is a non-zero divisor in
$\k[x_1,\dots,x_n]/(x_1^2,\dots,x_{i-1}^2)$. Setting $d = 2$, $e = 1$, $c = n-2$ and $r = n$ in Corollary~6.10~(2)
from~\cite{concaherzogtrungvalla97} we get that $\k[I_n]$ is Koszul where
$\k[I_n]$ is generated by all monomials of degree $n$ belonging to $I$; that
is, $\k[I_n] = \k[\A^\bullet_n]$. 

\medskip
Very recently Vu~\cite{vu13arXiv} proved a general result that for $m,n \geq 2$
and $\xx \in \A_{m,n}$ the ring $\k[\A_{m,n} \setminus \{\xx\}]$ is Koszul
unless $m \geq 3$ and $\xx$ is $(0,\dots,0,2,m-2)$ or one of its permutations
(this result also includes Corollary~\ref{c:koszul}).

\heading{Further related work.} Here we very briefly mention further related
work. We keep several terms undefined in this paragraph. The reader is welcome
to consult the cited sources for more details.
Eisenbud, Reeves and Totaro~\cite{eisenbud-reeves-totaro94} showed that the
$m$th Veronese subrings of $\k[z_1,\dots,z_t]/I$ are Koszul where $I$ is
a homogeneous ideal and $m$ is large enough (more precisely when $m \geq
\reg(I)/2$ where $\reg(I)$ is Castelnuovo-Mumford regularity of $I$). Further
investigation of Koszulness of $\k[z_1,\dots,z_t]/I$ can be found
in~\cite{batzies-welker02,hersh-welker05,herzog-reiner-welker98,peeva-reiner-sturmfels98} in the context
where the generators $z_i$ correspond to monomials $x^{\aa}$ as above and $I$
records the syzygies between the monomials (and then $\k[z_1,\dots,z_t]/I
\simeq \k[\A]$).

\heading{Structure.}
In Section~\ref{s:crit} we explain our new method for showing shellability. In
Section~\ref{s:crit_proofs} we prove the correctness of the method.
Section~\ref{s:Veronese} serves as a preliminary section on properties of
the (pinched) Veronese posets. In Section~\ref{s:proof_vn} we prove
Theorem~\ref{t:vn}. Finally, in Section~\ref{s:lex} we compare the strength of
our shellability method (mainly) with standard chain-lexicographic shellability
of Bj\"{o}rner and Wachs~\cite{bjorner-wachs82}. If the reader is more
interested in the shellability criteria rather than Theorem~\ref{t:vn}, we
highly recommend to read Section~\ref{s:lex} right after
Section~\ref{s:crit}. Here we offer the graph of the dependency of the
sections:
\begin{center}
\includegraphics{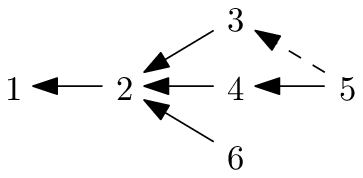}
\end{center}
The dashed arrow between Sections~\ref{s:crit_proofs} and~\ref{s:proof_vn}
means that Section~\ref{s:crit_proofs} is not necessary for understanding
Section~\ref{s:proof_vn}; however, the correctness of the proof in
Section~\ref{s:proof_vn} is based on Section~\ref{s:crit_proofs}.

\section{Method for showing shellability}
\label{s:crit}
In this section we describe our main tools for the proof of Theorem~\ref{t:vn}.
We need to set up some preliminaries first.

\heading{Poset preliminaries.}
Let $P = (P,\leq)$ be a graded poset with rank function $\rk$. By $\hat 0$ we
mean the unique minimal element of $P$ (if it exists) and similarly by $\hat 1$ we
mean the unique maximal element (if it exists). For $a, b \in P$ we say that $a$
\emph{covers} $b$, $a \covers b$, if $a > b$ and there is no $c$ with $a > c >
b$. Equivalently, $a > b$ and $\rk(a) = \rk(b) + 1$. Pairs of elements $a, b$
with $a \covers b$ are also known as \emph{edges} in the Hasse diagram of $P$.
\emph{Atoms} are elements that cover $\hat 0$. That is, atoms are elements of
rank $1$ in a poset that contains $\hat 0$. 

From now on, let us assume that $P$ contains a unique minimal element. 
Let $A$ be a set of some atoms in $P$. 
By $P\hull{A} = (P\hull{A}, \leq)$ we mean the induced subposet of $P$ with the ground set
$$
P\hull{A} = \{\hat 0\} \cup \{b \in P\colon b \geq a \hbox{ for some }a \in
A\}.
$$
%
%

\heading{Shellability.}
Now we assume that $P$ contains both a unique minimal and a unique maximal element. 
Let $C(P)$ be the set of maximal chains of $P$. 
A \emph{shelling order} is an order of chains from $C(P)$ satisfying the
following condition.

\begin{itemize}
  \item[(Sh)] If $c'$ and $c$ are two chains from $C(P)$ 
  such that $c'$ appears before $c$, then there is a chain $c^*$ from $C(P)$
  appearing before $c$ such that $c \cap c^* \supseteq c \cap c'$ and also 
  $c$ and $c^*$ differ in one level only (that is, $|c \Delta c^*| = 2$ where
  $\Delta$ denotes the symmetric difference).
\end{itemize}

A poset $P$ is \emph{shellable} if it admits a shelling order.
This is equivalent with saying that the order complex of $P$ is
shellable (as a simplicial complex).

\heading{$A$-shellability.}
Now let us assume that $A = (A, \aleq)$ is a partially ordered set of some atoms in $P$. We say
that $P\hull{A}$ is \emph{$A$-shellable} if $P\hull{A}$ is shellable with a shelling
order respecting the order on $A$. That is, if $c$ and $c'$ are two maximal
chains on $P\hull{A}$ and the unique atom of $c'$ appears before the unique atom
of $c$ in the $\aleq$ order, then $c'$ appears before
$c$ in the shelling.\footnote{For purposes of Theorem~\ref{t:vn}, it would be
fully sufficient to consider $\aleq$ as a linear order (a.k.a. total order).
However, we use partial orders, because nothing new has to be done to obtain more general
criteria with partial orders; and we believe that for some further applications
partial orders might be important.}

\heading{Using $A$-shellability.}
Let $P$ be a poset for which we aim to show that $P$ is shellable (in our
application $P = \pinch_n$). Let us order all the atoms of $P$ into a sequence 
$a_1, \dots, a_t$. For $k \in [t]$ let us set $A_k := \{a_1, \dots, a_k\}$ and
consider $A_k$ as a partially ordered set with the order $a_1 \aleq a_2 \aleq
\cdots \aleq a_k$. We would like to prove inductively that $P\hull{A_k}$ is
$A_k$-shellable. Let us assume that we are able to perform the first induction
step, that is, to show $A_1$-shellability of $P\hull{A_1}$ and let us focus on
the second induction step. We will provide two criteria, Theorems~\ref{t:ci}
and~\ref{t:cii} below, how to prove $A_{k+1}$-shellability of $P\hull{A_{k+1}}$
assuming $A_{k}$-shellability of $P\hull{A_k}$. 

This technique is quite similar to the technique using \emph{recursive atom
orderings} defined by Bj\"{o}rner and Wachs~\cite{bjorner-wachs83} and a
comparison of these two techniques is discussed in Section~\ref{s:lex}. In particular, the second criterion
(Theorem~\ref{t:cii}) is set up in such a way that it covers the case of
recursive atom orderings.
However, the technique presented here allows more freedom. In particular it
allows to combine different criteria to achieve the task.

One technical issue is the following. In our application for the pinched Veronese posets, it is not enough to consider the induction steps along a single ordering $a_1
\aleq a_2 \aleq \cdots \aleq a_t$ of the atoms of $P$. If we aimed on a single
ordering only, we would not have strong enough induction assumption to achieve
the task. Thus we will rather focus on many orderings of the atoms. For
considering more orderings simultaneously, it pays off to set up a third
criterion, Theorem~\ref{t:ciii}, which allows to `restrict' an $A$-shelling of
$P\hull{A}$ to an $A'$-shelling of $P\hull{A'}$ where $A'$ is a subset of $A$. 

\heading{Necessity of the criteria.} In our approach, the first criterion,
Theorem~\ref{t:ci}, seems to be the most important. The remaining two theorems
could, perhaps, be circumvented; however, they will simplify our analysis.



\medskip

\heading{Setting up the criteria.}
To set up conditions in the criteria, we need some additional notation. 
We fix some partially ordered set $A = (A, \aleq)$ of atoms of $P$ and a
further atom $a^+$ which is not in $A$. Think of $A = A_k$ and $a^+ =
a_{k+1}$ when comparing with the sketch above (it is more convenient to use a
notation independent of the index $k$).

%
We set $A^+ := A \cup \{a^+\}$ and $Q := P\hull{A^+} \setminus
P\hull{A}$. The partial order on $A^+$, which we again denote by $\aleq$,
extends $\aleq$ on $A$ so that $a^+ \ageq a$ for any $a \in A$.
We also consider $Q = (Q, \leq)$ as a subposet of $P$ with the unique minimal element
$a^+$ (it does not need to have a unique maximal element).

For $q \in Q$, we set $I(q)$ to be the interval $[q, \hat 1]$. Elements of
$P$ that cover $q$ are atoms of $I(q)$. By $A(q)$ we denote the set of (all)
atoms of $I(q)$ which simultaneously belong to $P\hull{A}$. By $\Aall(q)$ we
denote the set of all atoms of $I(q)$.
In particular,
note that the poset $I(q)\hull{A(q)}$ is well defined (we will need this poset later
on).

%

\heading{Edge falling property.}
Let $q \in Q$. We say that $q$ has the \emph{edge falling} property if for
every $p \in P\hull{A}$ with $p \covers q$
and every $q' \in Q \cup \{\hat 0 \}$ with $q \covers q'$ there is $p' \in
P\hull{A}$ such that $p \covers p' \covers q'$. See
Figure~\ref{f:ef}.
\begin{figure}
\begin{center}
  \includegraphics{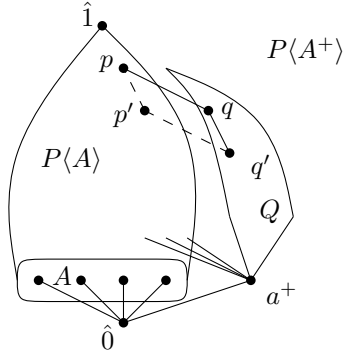}
\end{center}
\caption{The edge falling property. The $P\hull{A}$--$Q$ edge $pq$ falls by one
level to $p'q'$.}
\label{f:ef}
\end{figure}

%
%
%

%
%

\heading{Shellability criteria.} Now, we can state our first criterion; see also Figure~\ref{f:ci}.

\begin{theorem}[Criterion I]
\label{t:ci}
  The poset $P\hull{A^+}$ is $A^+$-shellable if the following conditions are
  satisfied.
  \begin{enumerate}[(i)]
    \item $P\hull{A}$ is $A$-shellable;
    \item for every $q \in Q$ the interval $[a^+,q]$ is shellable;
    \item every $q \in Q$ has the edge falling property; and
    \item for every $q \in Q$ the poset $I(q)\hull{A(q)}$ is shellable.
  \end{enumerate}
\end{theorem}

\begin{figure}
\begin{center}
  \includegraphics{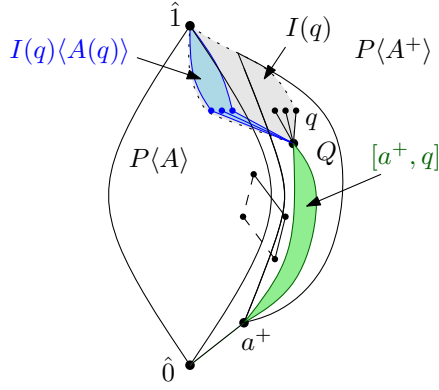}
\end{center}
\caption{Important subposets appearing in the conditions of Theorem~\ref{t:ci}. We
also recall the edge-falling property by a little diamond between $P\hull{A}$
and $Q$.}
\label{f:ci}
\end{figure}

%
%
%

The second criterion is similar to the first one; however, it focuses more on
the structure of the interval $I(a^+)$ rather than on the structure of $Q$. See
also Figure~\ref{f:cii}.

\begin{theorem}[Criterion II]
\label{t:cii}
  The poset $P\hull{A^+}$ is $A^+$-shellable if the following conditions are
  satisfied.
  \begin{enumerate}[(i)]
    \item $P\hull{A}$ is $A$-shellable;
    \item there is a linear order on $\Aall(a^+)$ such that the elements of
      $A(a^+)$
      appear before other elements in this order and such that $I(a^+) =
      I\hull{\Aall(a^+)}$ is
      $\Aall(a^+)$-shellable (with respect to this order); and
    \item for every $q \in Q$ and for every $p \in P\hull{A}$ if $p \covers q$,
      then $p \in I(a^+)\hull{A(a^+)}$.
  \end{enumerate}
\end{theorem}

\begin{figure}
\begin{center}
  \includegraphics{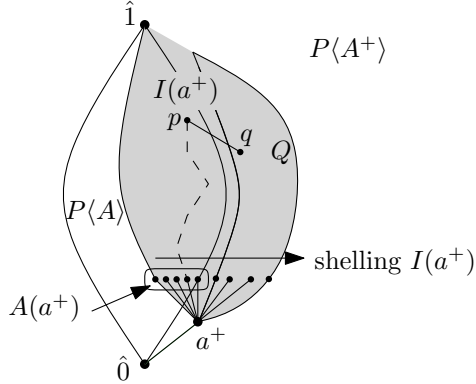}
\end{center}
\caption{Schematic drawing of the conditions of Theorem~\ref{t:cii}.}
\label{f:cii}
\end{figure}

The third criterion that we provide below differs from the previous two. In
this case we rather reduce $A$ to $A'$ instead of enlarging it.

\begin{theorem}[Criterion III]
\label{t:ciii}
  Let $A'$ be a subset of $A$, linearly ordered with the order inherited from
  $A$.  The poset $P\hull{A'}$ is $A'$-shellable if the following conditions
  are satisfied.
 \begin{enumerate}[(i)]
   \item $P\hull{A}$ is $A$-shellable; and
   \item for every $b \in A \setminus A'$ and for every $p \in P\hull{A'}$ with
     $p \covers b$, there is $b'$ appearing before $b$ in $A$ such that $b' \in
     A'$ and $p \covers b'$ (see Figure~\ref{f:ciii}).
 \end{enumerate}
\end{theorem}

\begin{figure}
\begin{center}
  \includegraphics{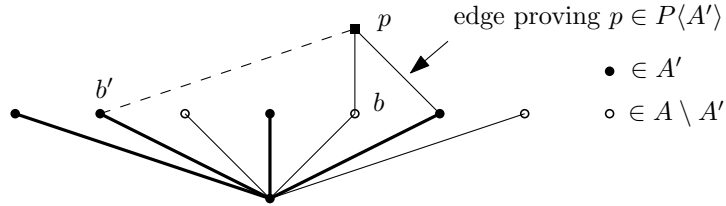}
\end{center}
\caption{Schematic drawing of condition (ii) of Theorem~\ref{t:ciii}.}
\label{f:ciii}
\end{figure}

The proofs of all three criteria are given in Section~\ref{s:crit_proofs}.

We conclude this section by remarks about the differences in the criteria
above and their comparison to lexicographic shellability.

\heading{Relation between Criterion I and Criterion II.}
A reader might check that Theorem~\ref{t:ci} `almost' follows from
Theorem~\ref{t:cii}. More precisely, it is not hard to see that conditions~(i)
and~(iii) of Theorem~\ref{t:cii} easily follow from the assumptions of
Theorem~\ref{t:ci}. The main difference is that condition~(ii) of
Theorem~\ref{t:cii} does not immediately follow from the assumptions of
Theorem~\ref{t:ci}. (Assuming that the conditions of Theorem~\ref{t:ci} are
satisfied, we can immediately deduce that $I(a^+)\hull{A(a^+)}$ is shellable by
setting $q = a^+$ in condition~(iv) of Theorem~\ref{t:ci}; however, we do not
have shelling of whole $I(a^+)$ yet).

Actually, the essence of the proof of Theorem~\ref{t:ci} can be seen as
verifying condition~(ii) of Theorem~\ref{t:cii} from conditions~(ii), (iii) and
(iv) of Theorem~\ref{t:ci}, which is solely a property of a certain
decomposition of the interval $I(a^+)$. The interested reader is welcome to
formulate the criteria on extension of a shelling of $I(a^+)\hull{A(a^+)}$ to a
shelling of whole $I(a^+)$ separately, following the proof of Theorem~\ref{t:ci}.


\heading{Relation of lexicographic shellability and $A$-shellability.}
A very standard notion for showing that a certain poset is shellable is the 
so-called (chain lexicographic) CL-shellability, introduced by Bj\"orner and
Wachs~\cite{bjorner-wachs82}, or even a further generalization, the so-called 
(chain compatible) CC-shellability introduced by Kozlov~\cite{kozlov97}, still induced
by a lexicographic order on chains. It is natural to ask
what is the relation between $A$-shellability introduced here and lexicographic
shellability (we will focus on CL-shellability only; some ideas can be carried
for CC-shellability as well). 
We discuss this relation in more detail in separate
Section~\ref{s:lex} and the reader interested in these details is encouraged to
read Section~\ref{s:lex} immediately (perhaps after finishing this section). Questions addressed in Section~\ref{s:lex} have 
arisen in discussions with Anders Bj\"{o}rner and Afshin Goodarzi. Here we
briefly survey these questions. 

It is not hard to see
that every lexicographically shellable poset is $A$-shellable where $A$ is set
of all atoms equipped with an appropriate linear order. On the other hand, it
is not hard to find an $A$-shellable poset (again with $A$ consisting of all
atoms) which is not lexicographically shellable.

We can also ask more subtle questions about the relative power of
Theorems~\ref{t:ci} and~\ref{t:cii} compared with lexicographic shellability.
(We skip Theorem~\ref{t:ciii} since it is of a different spirit.)

The conditions of Theorem~\ref{t:cii} are analogous to the conditions on recursive atom
orderings from~\cite{bjorner-wachs83}; and in particular Theorem~\ref{t:cii}
preserves lexicographic shellability (if the `shellable' assumptions are
changed into `lexicographically shellable') as well as lexicographically
shellable posets satisfy the conditions of Theorem~\ref{t:cii}. The added value of
Theorem~\ref{t:cii} appears when we use it with non-lexicographic assumptions.

Regarding Theorem~\ref{t:ci} let us (again) consider the following two questions: whether a lexicographically
shellable poset satisfies the criteria of Theorem~\ref{t:ci}; and whether
lexicographic shellability is kept by the criteria of Theorem~\ref{t:ci} (for linearly ordered $A$).

The answer to the first question is no. The answer to the second question is
not known to the author. We just remark that the proof of Theorem~\ref{t:ci}
might produce non-lexicographic shelling even if all posets in the conditions of
Theorem~\ref{t:ci} are assumed to be lexicographically shellable (not even a
CC-shelling). We again
refer to Section~\ref{s:lex} for more details.

The above-mentioned remarks suggest that $A$-shellability using
Theorem~\ref{t:ci} and lexicographic shellability are perhaps 
in `generic position' regarding applicability in various situations.

\section{Proofs of shellability criteria}
\label{s:crit_proofs}

Here we prove Theorems~\ref{t:ci},~\ref{t:cii}, and~\ref{t:ciii}. We keep the notation introduced in the previous section. 

Below we also set up an additional notation common to proofs of Theorems~\ref{t:ci}
and~\ref{t:cii}.
Let $C := C(P\hull{A})$ and $C^+ = C(P\hull{A^+})$ be the sets of maximal
chains in $P\hull{A}$ and $P\hull{A^+}$.
We know that $P\hull{A}$ is shellable, therefore there is some shelling 
order $c_1, c_2, \dots, c_t$ of all chains from $C$ (note that $P\hull{A}$
contains both $\hat 0$ and $\hat 1$). We are going to describe a shelling order
on $C^+$. In both cases, we start with $c_1, \dots, c_t$ and then we continue
with chains containing $a^+$. This way, if we show that we have a shelling order,
it will immediately be an $A^+$-shelling.

\subsection{Proof of Theorem~\ref{t:ci}}

We choose some order $q_1, \dots, q_u$ of elements of $Q$ such that
$i \leq j$ if $\rk(q_i) \leq \rk(q_j)$. In particular $q_1 = a^+$.
For every $q_i \in Q$ we have
an order of maximal chains in the interval $[a^+,q_i]$ inducing a shelling of
this interval, by condition~(ii).

Now we describe a shelling order of all maximal chains from $C^+\setminus C$.
(We already have an order on $C$.) Let
$c$ be a chain from $C^+ \setminus C$, the index $i(c)$ is denoted in such a way that
$q_{i(c)}$ is the element of $c \cap Q$ with the largest rank. Note that if
$r \in c$, $r \neq \hat 0$, and $\rk(r) < \rk(q_{i(c)})$, then $r \in Q$.

Now let $c$ and $c'$ be two different chains from $C^+ \setminus C$ and we want to
describe when $c'$ is before $c$.

The first criterion is whether $i(c') < i(c)$. That is, if $i(c') < i(c)$, then
$c'$ is sooner in the order than $c$ (and symmetrically $c'$ is later if $i(c')
> i(c)$); see Figure~\ref{f:i_order}, on the left.

\begin{figure}
\begin{center}
  \includegraphics{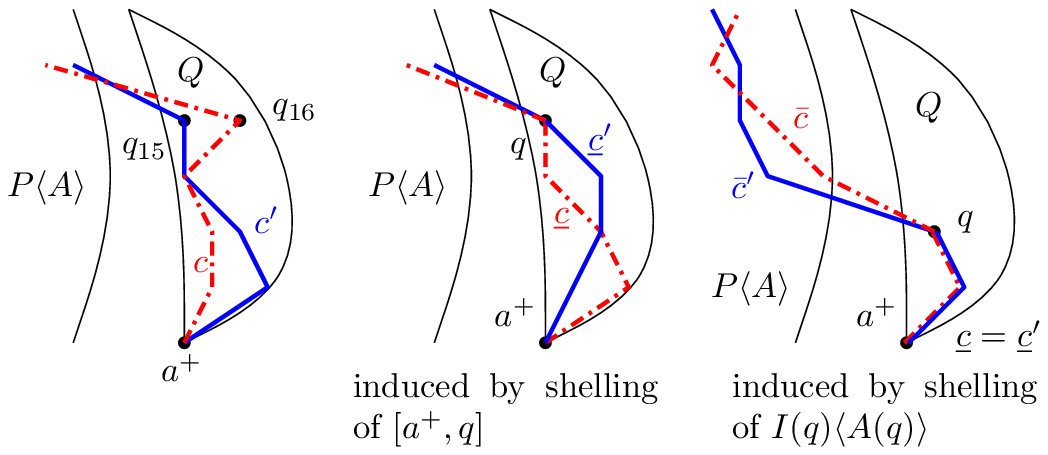}
\end{center}
\caption{Three cases when $c'$ appears before $c$.}
\label{f:i_order}
\end{figure}

If $i(c) = i(c')$, then we have the following second criterion. Let $q =
q_{i(c)} = q_{i(c')}$. We look at the
two maximal chains $\underline c = c \cap [a^+,q]$ and $\underline c' = c' \cap
[a^+,q]$ in
the interval $[a^+,q]$. As we sooner realized, if $\underline c \neq \underline c'$, then
there is order of these chains inducing a shelling on $[a^+,q]$. This induces
the order of $c$ and $c'$; see Figure~\ref{f:i_order}, in the middle. If $\underline c = \underline c'$, we need a third criterion.

Now we assume that $i(c) = i(c')$ and $\underline c = \underline c'$. The element
$q$ is defined as above. We set $\bar c = c \cap I(q)$ and $\bar c' = c'
\cap I(q)$ recalling that $I(q)$ is the interval $[q, \hat 1]$. Both chains
$\bar c$ and $\bar c'$ are maximal chains in $I(q)\hull{A(q)}$ due to the
choice of $q = q_{i(c)} = q_{i(c')}$. The condition (iv) in the statement of the theorem implies
that $I(q)\hull{A(q)}$ is shellable. We set
that $c'$ appears before $c$ in our shelling if and only if $\bar c'$ appears
before $\bar c$ in the shelling of $I(q)\hull{A(q)}$; see
Figure~\ref{f:i_order}, on the right.


%

\medskip

We have described an order of chains in $C^+$. Now we have to prove that it is
indeed a shelling order. That is we have to prove condition (Sh).
In the sequel we therefore assume that $c$ and $c'$ are given, as in (Sh), and we
seek for $c^*$. 

If $c \in C$, then we find required $c^*$ immediately from shellability of
$P\hull{A}$. In the sequel we assume $c \in C^+ \setminus C$ and we set $q =
q_{i(c)}$. We distinguish several cases.

\begin{enumerate}
%
  
  \item $q \not\in c'$.

    In this case we use the edge falling property. Let $q'$ be the element of
    $c$ such that $q \covers q'$ and $p$ be the element of $c$ such that $p
    \covers q$. The edge falling property 
implies that there is $p' \in P\hull{A}$ such that $p \covers p' \covers q'$. We set up $c^* = (c \cup
\{p'\}) \setminus \{q\}$. Obviously, $c^*$ satisfy the required properties.

  \item $q \in c'$, and $\underline c \neq \underline
    c'$ (where $\underline c = c \cap [a^+,q]$
    and $\underline c' = c' \cap [a^+,q]$).
  
    By their definition, $c'$ appears before $c$, thus due to the first
    criterion we have that $i(c') \leq i(c)$. Now since $q \in c'$, it follows
  that $i(c') = i(c)$ and therefore $q = q_{i(c')}$ (that is $q$ is
    the element of $c' \cap Q$ of the highest rank). In addition, due to the second criterion, we
    know that
  $\underline c'$ appears before $\underline c$ in the shelling of $[a^+,q]$. Therefore
  there is a maximal chain $\underline c^*$ in $[a^+,q]$ appearing before
  $\underline c$ 
  which coincides with $\underline c$ with exception of one level and such that
  $\underline
  c \cap \underline c^* \supseteq \underline c \cap \underline c'$. We set $c^*$ so
  that it coincides with $\underline c^*$ on $Q$ and with $c$ on $P\hull{A}$.

  \item $q \in c'$, and $\underline c = \underline c'$.

    We again have $q = q_{i(c')}$. Hence, the third criterion on comparison of $c$ and $c'$ applies.
   That is, $\bar c'$ appears before $\bar c$ in the shelling of
   $I(q)\hull{A(q)}$. Similarly as in the previous case, there is, therefore, a
   maximal chain $\bar c^*$ in $I(q)\hull{A(q)}$ appearing before $\bar c$
   which coincides with $\bar c$ with exception of one level and such
   that $\bar
         c \cap \bar c^* \supseteq \bar c \cap \bar c'$ (recall that $\bar c =
	 c \cap I(q)$ and $\bar c' = c' \cap I(q)$).
  We set $c^*$ so that it coincides with $c$ on $Q$ and with $\bar c^*$ on $P\hull{A}$.
\end{enumerate}

We have verified condition~(Sh) in all cases. This concludes the proof of
Theorem~\ref{t:ci}.

\subsection{Proof of Theorem~\ref{t:cii}}
In this case, it is easier to set up the order of shelling $C^+\setminus C$.
(Let us recall that the order on $C$ is already set up, and that the chains
from $C^+ \setminus C$ will follow after the chains from $C$.)

Every chain $c \in C^+ \setminus C$ contains $a^+$. Let $\bar c$ be in this case
$c \cap I(a^+)$. We set that $c'$ precedes $c$ if and only if $\bar c'$ precedes
$\bar c$ in the shelling from condition (ii) of the statement of the theorem.

Now, we need to verify condition (Sh) to be sure that we have indeed a shelling
order. Similarly as in the proof of previous theorem, we assume that $c$ and $c'$ are given, as in (Sh), and we
seek for $c^*$. We distinguish several cases.

\begin{enumerate}
  \item $c \in C$.
  
   In this case we know that $c'$ appears before $c$ and thus $c' \in C$.
   Therefore, we can find suitable $c^*$ from the shellability of $P\hull{A}$.

  \item $c \in C^+ \setminus C$ and $c' \in C^+ \setminus C$.

   In this case $\bar c'$ appears before $\bar c$, therefore, there is $\bar
   c^*$ from shelling of $I(a^+)$ such that $\bar c$ and $\bar c^*$ differ in one
   level only and that $\bar c^* \cap \bar c \supseteq \bar c' \cap \bar c$. We
   set $c^* = \bar c^* \cup \{\hat 0\}$. This choice of $c^*$ obviously satisfy
   the required properties.

  \item $c \in C^+ \setminus C$, $c' \in C$, and $c \cap A(a^+) \neq \emptyset$.

  Let $b \in c \cap A(a^+)$. Then there is $a \in A$ such that $b \covers a$
  due to the definition of $A(a^+)$. Let us set $c^* := (c \setminus \{a^+\}) \cup
  \{a\}$. Then $c^* \cap c \supseteq c' \cap c$ since $c'$ misses $a^+$. See
  Figure~\ref{f:ii_cstar}, on the left.

  \item $c \in C^+ \setminus C$, $c' \in C$, and $c \cap A(a^+) = \emptyset$.

    As usual, let $q$ be the largest element of $c \cap Q$. Let $p$ be the
    element of $c \cap P\hull{A}$ such that $p \covers q$. See
    Figure~\ref{f:ii_cstar}, on the right. Condition (iii) in
    the statement of the theorem implies that there is a maximal chain $
    c'_2$ in the interval $[a^+,p]$ such that $c'_2 \cap A(a^+) \neq \emptyset$.
    Let $\bar c'_2$ be the maximal chain in $I(a^+)$ which agrees with $c'_2$
    on $[a^+,p]$ and which agrees with $c$ on $[p,\hat 1]$. Note that $\bar
    c'_2$
    precedes $\bar c$ in the shelling of $I(a^+)$ since $\bar c'_2 \cap A(a^+) \neq
    \emptyset$ whereas $c \cap A(a^+) = \emptyset$. Therefore, by~(Sh), there is
    a chain $\bar c^*$ in $I(a^+)$ which agrees with $\bar c$ in all levels but
    one and which satisfies $\bar c^* \cap \bar c \supseteq \bar c'_2 \cap c$. In
    particular $\bar c^*$ agrees with $\bar c$ on $p$ and all elements above
    $p$. Now, we set $c^* := \bar c^* \cup \{\hat 0\}$. We have that $c^* \cap
    c \supseteq c' \cap c$ since $c' \cap c \subseteq P\hull{A}$.
\end{enumerate}

\begin{figure}
\begin{center}
  \includegraphics{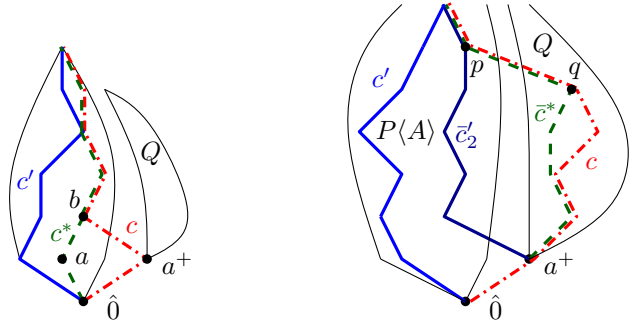}
\end{center}
\caption{Cases 3 and 4 in the proof of Theorem~\ref{t:cii}.}
\label{f:ii_cstar}
\end{figure}

This finishes the proof of Theorem~\ref{t:cii}.

\subsection{Proof of Theorem~\ref{t:ciii}}

Let $C = C(P\hull{A})$ and $C' = C(P\hull{A'})$ be the sets of maximal chains of
$P\hull{A}$ and $P\hull{A'}$. We have that $C' \subset C$. Since $P\hull{A}$ is
$A$-shellable, we have a shelling order on $C$ respecting $A$. We simply
inherit this order on $C'$. It respects $A'$; however, we have to show that it
is indeed a shelling order.

Let $c$ and $c'$ be chains in $P\hull{A'}$ such as in condition~(Sh). We look 
for a suitable $c^*$ from (Sh).

Chains $c$ and $c'$ also belong to $P\hull{A}$. Since we started with a
shelling on $C$, there is $c^{**} \in C$ such that $c^{**} \cap c \supseteq
c' \cap c$ and $c^{**}$ differs from $c$ in one level. If $c^{**}$ belongs to
$C'$, we set $c^* := c^{**}$ and we are done.

Now let us assume that $c^{**} \not \in C'$. Let $b$ and $p$ be the elements of
$c^{**}$ of rank 1 and 2 respectively, in particular $p \covers b$. Since
$c^{**} \notin C'$, it follows from the definition of $C'$ that $b \in A
\setminus A'$. Moreover, $c$ and $c^{**}$ differ in only one level. Therefore
they differ in level 1 and $p \in c$. This implies that $p \in P\hull{A'}$. By
applying now assumption~(ii) of the theorem for elements $b$ and $p$ we conclude
that there is $b' \in A'$ appearing before $b$ in $A$ such that $p \covers b'$. Let us set
$c^* := (c^{**} \setminus \{b\}) \cup \{b'\}$. Then $c^*$ appears before
$c^{**}$ in the shelling of $C$ and hence also before $c$. In addition $c^{*}$ and $c$ have to differ in level $1$
(only) by definition of $c^*$. Thus we
obtain $c^* \cap c = c^{**} \cap c \supseteq c' \cap c$ as required.


This finishes the proof of Theorem~\ref{t:ciii}.

\section{Preliminaries on the (pinched) Veronese poset}
\label{s:Veronese}

The $n$-th \emph{Veronese} poset $(\ver_n, \leq )$ is given by
$$
\ver_n = \{(\alpha_1, \dots, \alpha_n) \in \Nz^n \colon \alpha_1 + \cdots +
  \alpha_n \equiv 0
\pmod n\}
$$
and $\aa \leq \bb$ for $\aa = (\alpha_1, \dots, \alpha_n)$, $\bb = (\beta_1,
\dots, \beta_n)$ if
and only if $\alpha_i \leq \beta_i$ for $i \in [n]$. In the sequel, we often write $\aa =
\alpha_1\alpha_2\alpha_3$ instead of $\aa = (\alpha_1, \alpha_2, \alpha_3)$ and
so on for higher $n$. We can also use brackets to separate coordinates in
expressions such as $(\alpha_1 + 1)01\alpha_4$ instead of $(\alpha_1 + 1, 0, 1,
\alpha_4)$.

In slightly more general setting, for positive integers $m$ and $n$ we also
define
$$
\ver_{m,n} = \{(\alpha_1, \dots, \alpha_n) \in \Nz^n \colon \alpha_1 + \cdots +
     \alpha_n \equiv 0
   \pmod m\}.
$$
We again have that $\aa \leq \bb$ if $\aa$ is less or equal to $\bb$ in every
coordinate. In particular, we have $\ver_n = \ver_{n,n}$.

The $n$-th \emph{pinched Veronese} poset $(\pinch_n, \preceq)$ is a (non-induced) 
subposet of $\ver_n$ given by the following data.
$$
\pinch_n = \{\aa \in \ver_n \colon \aa \neq \jj\}.
$$
Here $\jj = 1\cdots 1$.
The partial order on $\pinch_n$ is given by $\aa \preceq \bb$ if $\aa \leq \bb$ and
$\bb - \aa \neq \jj$.

We also define $\OO = 0\cdots0$ to be the minimal element of $\pinch_n$.

\heading{Arithmetic operations on $\ver_n$ and $\pinch_n$.} We consider elements
of $\ver_n$ and $\pinch_n$ as vectors in $\Z^n$. We can then sum and subtract
these vectors. For a set $X \subseteq \Z^n$ and vector $\vv \in \Z^n$ we let $X
\oplus \vv$ to be the set $\{\xx + \vv \colon \xx \in X\}$. Similarly, $X
\ominus \vv := \{\xx - \vv\colon \xx \in X\}$. Let $[\OO,\zz]$ be
an interval in $\pinch_n$ and $\xx \in [\OO,\zz]$. In our considerations, we
will often use the fact that $[\xx, \zz]$ and $[\OO,\zz - \xx]$ are isomorphic;
more precisely, $[\OO,\zz - \xx] = [\xx,\zz] \ominus \xx$.

\heading{Shellability of intervals in $\ver_{m,n}$.}
It is not hard to observe, using known results, that every interval in
$\ver_{m,n}$ is shellable. We will actually need this for
considering the pinched version, thus we provide full details.

\begin{proposition}
  \label{p:non_pinch}
  Let $m$ and $n$ be positive integers. For any $\zz \in \ver_{m,n}$ the
  interval $[\OO,\zz]$ in $\ver_{m,n}$ is a shellable poset.
\end{proposition}

\begin{proof}
  We have that $\ver_{m,n}$ is a subposet of $\ver_{1,n}$. We first observe
  that $[\OO, \zz]$ is shellable as an interval in $\ver_{1,n}$ and then we
  deduce that $[\OO, \zz]$ is shellable as an interval in $\ver_{m,n}$ as well.

  It is not hard to observe that $[\OO, \zz]$ as an interval in 
  $\ver_{1,n}$ is a graded modular lattice: By
  \emph{modular} we mean that
  $$
  \rk(\aa) + \rk(\bb) = \rk(\aa \vee \bb) + \rk(\aa \wedge \bb).
  $$
  If $\aa = \alpha_1\cdots\alpha_n$ and $\bb = \beta_1\cdots\beta_n$, then 
  $$\aa \vee \bb = \max(\alpha_1,\beta_1)\cdots\max(\alpha_n,\beta_n)$$
  and
  $$\aa \wedge \bb = \min(\alpha_1,\beta_1)\cdots\min(\alpha_n,\beta_n).$$
  These relations easily imply modularity of $\ver_{1,n}$. Therefore,
  $\ver_{1,n}$ is shellable by~\cite[Theorem 3.7]{bjorner80}
  (\emph{semimodular} would be sufficient).

  The fact that $\ver_{m,n}$ is shellable follows from the fact that the
  shellability is preserved by rank-selections. Indeed, if we start with $[\OO,
  \zz]$ as an interval in $\ver_{1,n}$ we remove elements exactly in levels not
  divisible by $m$ in order to turn it into an interval in $\ver_{m,n}$. This
  means that we remove the same number of elements from every maximal chain.
  Therefore, $\ver_{m,n}$ is shellable by~\cite[Theorem 11.13]{bjorner95}.

\end{proof}

\section{Proof of Theorem~\ref{t:vn}}
\label{s:proof_vn}

The task of this section is to prove Theorem~\ref{t:vn}. 
Throughout this section we assume that $n \geq 4$ is fixed.


\subsection{The induction mechanism}

Let $\Aall$ be the set of all atoms of
$\pinch_n$.\footnote{It can be computed that $|\Aall| = \binom{2n-1}n - 1$;
however, we will not need to know this value explicitly.}
We will consider several linear orders on $\Aall$ and some of its subsets. Let $\xx =
\xi_1\cdots\xi_n \in \Z^n$. For $\ell \in [n]$ we set  $\xx^{(\ell)} =
\xi_{\ell}
\cdots \xi_n$. We also set $A^{(\ell)}$ to be the subset of $\Aall$ made of all
$\xx \in \Aall$ such that $\xx^{(\ell)} \neq 0\cdots 0$. We consider two linear
orders, $<^L$ and $<^S$ on $\Aall$. 

The first order is the lexicographic order given in the following way.
Let $\ss = \sigma_1\cdots\sigma_n$ and $\ttt = \tau_1\cdots\tau_n$. 
We set $\ss <^L \ttt$ if and only there is $j \in [n]$ such
that $\sigma_i = \tau_i$ for $i < j$ and $\sigma_j < \tau_j$. 

\begin{table}
\begin{center}
\begin{tabular}{cccccccccc}
  $<^L$ on $\Aall$: &  0004 & 0013 & 0022 & 0031 & 0040 & 0103 & 0112 & 0121 \cr
				  
& 0130 & 0202 & 0211 & 0220 & 0301 & 0310 & 0400 & 1003 \cr
& 1012 & 1021 &  1030 & 1102 & \xcancel{1111} & 1120 & 1201 & 1210 \cr
& 1300 & 2002 & 2011 &  2020 & 2101 & 2110 & 2200 & 3001 \cr
& 3010 & 3100 & 4000 \\[5mm]
\end{tabular}
\begin{tabular}{ccccccccc}
$<^L$ on $A^{(4)}$: & 0004 & 0013 & 0022 & 0031 & 0103 & 0112 & 0121 & 0202 \cr 
& 0211 & 0301 & 1003 & 1012 & 1021 & 1102 & \xcancel{1111} & 1201 \cr
& 2002 & 2011 & 2101 & 3001 \\[5mm]
\end{tabular}
\begin{tabular}{crccccccc}
$<^S$ on $\Aall$: & $A^S$ & 0004 & 0013 & 0022 & 0031 & 0103 & 0112 & 0121 \cr
& & 0202 & 0211 & 0301 & 1003 & 1012 & 1021 & \xcancel{1102} \cr 
& & \xcancel{1111} & 1201 & 2002 & 2011 & 2101 & 3001 \cr
& $\{1102\}$ & 1102 \cr
& $\Aall \setminus A^{(4)}$ & 0040 & 0130 & 0220 & 0310 & 0400 & 1030 & 1120
  \cr
  & & 1210 & 1300 & 2020 & 2110 & 2200 & 3100 & 3010 \cr 
  & & 4000 \cr
\end{tabular}
\caption{Atoms of $\Aall$ and $A^{(4)}$ sorted by the $<^L$ order and atoms of
$\Aall$ sorted by the $<^S$ order for $n= 4$.}
\label{tab:orders}
\end{center}
\end{table}

The second order is a $\emph{specific}$ order which we describe now. We
set $A^S := A^{(n)} \setminus \{1\cdots102\}$. The smallest elements in $<^S$
order are the elements of $A^S$ sorted lexicographically by the $<^L$ order. Then
the element $1\cdots102$ follows. Finally, the elements of $\Aall \setminus
A^{(n)}$ follow sorted again by the $<^L$ order. The reader is referred to
Table~\ref{tab:orders} for more concrete comparison of these orders (for $n =
4$). 

We will need to work with the following ordered sets. Let
$\aa^L_i$ be the $i$th
smallest element of $\Aall$ in the $<^L$ order and similarly $\aa^S_i$ be the
$i$th smallest
element in the $<^S$ order. We then set $A^L_k := \{\aa_1^L,\dots,\aa_k^L\}$
and $A^S_k := \{\aa_1^S,\dots,\aa_k^S\}$.
We also set $A^{(\ell)}_k$ to be the set of the first $k$ elements of
$A^{(\ell)}$ in the $<^L$ order (this time, we omit the superscript $L$ for
simpler notation).

\medskip

Now let $I = [\OO,\zz]$ be any interval in $\pinch_n$. Our task is to show that
$I$ is shellable. In order to explain our next step let us use the following
simplification of notation. Let $A$ be some set of atoms of $I$ equipped with
the $<^L$ order (resp. with the $<^S$ order). Instead of saying that
$I\hull{A}$ is $A$-shellable we say that $I\hull{A}$ is \lsh-shellable  (resp.
$I\hull{A}$ is \ssh-shellable). This simplifies the notation when our typical $A$ will
be of form $A^{(\ell)}_k \cap I$. In addition, it also explicitly emphasizes
whether $A$ is equipped with the $<^L$ order or the $<^S$ order.

Our task will be to prove the assertions below. The first two assertions depend
on $k \leq |\Aall|$. The third assertion depends on $\ell \in [n-1]$ and $k
\leq |A^{(\ell + 1)}|$.\\

\begin{tabular}{cl}
  $(\As^L_{k})$ & The poset $I\hull{A^L_k \cap I}$ is \lsh-shellable
  (if nonempty).\\[3mm]
  $(\As^S_{k})$ & The poset $I\hull{A^S_k \cap I}$ is \ssh-shellable
  (if nonempty).\\[3mm]
  $(\As^{(\ell + 1)}_k)$ & The poset $I\hull{A^{(\ell + 1)}_k \cap I}$ is
  \lsh-shellable.\\[5mm]
\end{tabular} 

\begin{proposition}
\label{p:vn}
  Let $I = [\OO,\zz]$ be any interval in $\pinch_n$. Then assertions
  $(\As^L_{k})$ and $(\As^S_{k})$ are valid for any positive integer $k \leq
  |\Aall|$ and assertion $(\As^{(\ell + 1)}_k)$ is valid for any $\ell \in
  [n-1]$ and any positive integer $k \leq |A^{(\ell + 1)}|$.
\end{proposition}

Theorem~\ref{t:vn} follows from the proposition by setting $k = |\Aall|$
in $(\As^L_k)$ (or $(\As^S_k)$).


The task is to prove Proposition~\ref{p:vn} by a double induction. 
The first (outer) induction is over $\rk(\zz)$. 
The second (inner) induction is
slightly unusual---we first prove $(\As^L_{k})$ by induction in $k$ (see
Lemmas~\ref{l:LS1}, \ref{l:LS2}, and~\ref{l:Lk} below), then we
prove $(\As^S_{k})$ by induction in $k$ (see 
Lemmas~\ref{l:LS1}, \ref{l:LS2}, and~\ref{l:Sk} below), finally, we prove $(\As^{(\ell +
1)}_k)$ already assuming $(\As^L_{k})$ directly with no induction (see
Lemma~\ref{l:Al} below).
The fact that we use the induction is also the reason why we need to prove all
assertions $(\As^L_{k})$, $(\As^S_{k})$, and $(\As^{(\ell +
1)}_k)$, although only $(\As^L_{k})$ is sufficient for deducing
Theorem~\ref{t:vn}. We need the induction assumption strong enough so that
the induction works well. 

We also remark that $I$ does not need to contain all atoms from $\Aall$ (for example, if the
first coordinate of $\zz$ is zero). This is why we need to consider, for
example, \lsh-shellability of $I\hull{A^L_k \cap I}$ instead of
(possibly expected) \lsh-shellability of $I\hull{A^L_k}$.

For improved readability, we decompose the induction step into several lemmas,
with different approaches on how to prove them.
From now on we assume that $\zz$ and $I = [\OO,\zz]$ are fixed. 

\begin{lemma}
\label{l:LS1}
  Let us assume that Proposition~\ref{p:vn} is valid for every interval
  $[\OO,\yy]$ with $\rk(\yy) < \rk(\zz)$. Then $I\hull{A^L_1 \cap I}$ is
  \lsh-shellable and $I\hull{A^S_1 \cap I}$ is
  \ssh-shellable (if they are nonempty), that is, $(\As^L_{1})$ and
  $(\As^S_{1})$ are valid.
\end{lemma}

\begin{lemma}
\label{l:LS2}
  Let us assume that Proposition~\ref{p:vn} is valid for every interval
  $[\OO,\yy]$ with $\rk(\yy) < \rk(\zz)$. Then $I\hull{A^L_2 \cap I}$ is
  \lsh-shellable and $I\hull{A^S_2 \cap I}$ is
  \ssh-shellable (if they are nonempty), that is, $(\As^L_{2})$ and  
    $(\As^S_{2})$ are valid.
\end{lemma}

\begin{lemma}
\label{l:Lk}
  Let $k \in \{3, \dots, |\Aall|\}$. Let us assume that Proposition~\ref{p:vn} is valid for every interval
  $[\OO,\yy]$ with $\rk(\yy) < \rk(\zz)$. Let us also assume that
  $(\As^L_{k'})$ is valid for the interval $I = [\OO,\zz]$ and for $k' < k$.
  Then $I\hull{A^L_k \cap I}$ is
  \lsh-shellable (if nonempty), that is, $(\As^L_{k})$ is valid.
\end{lemma}

\begin{lemma}
  \label{l:Sk}
  Let $k \in \{3, \dots, |\Aall|\}$. Let us assume that Proposition~\ref{p:vn} is valid for every interval
  $[\OO,\yy]$ with $\rk(\yy) < \rk(\zz)$.
  Let us also assume that $(\As^S_{k'})$ is valid for the interval
  $[\OO,\zz]$ and for $k' < k$. 
  Then $I\hull{A^S_k \cap I}$ is \ssh-shellable (if nonempty), that is,
  $(\As^S_{k})$ is valid.
\end{lemma}

\begin{lemma}
\label{l:Al}
Let $\ell \in [n-1]$ and $k \in \{1, \dots, |A^{(\ell+1)}|\}$. 
Let us assume that $(\As^L_{k'})$
  is valid for the interval $I = [\OO,\zz]$ and for $k' = |\Aall|$.
  Then $I\hull{A^{(\ell+1)}_k \cap I}$ is
  \lsh-shellable (if nonempty), that is, $(\As^{(\ell+1)}_{k})$ is valid.
\end{lemma}

We remark that Lemma~\ref{l:LS2} implies Lemma~\ref{l:LS1}. Similarly,
Lemmas~\ref{l:Lk} and~\ref{l:Sk} together imply Lemma~\ref{l:LS2}. The reason
why we state Lemmas~\ref{l:LS1} and~\ref{l:LS2} separately is that
Lemma~\ref{l:LS1} is used in the proof of Lemma~\ref{l:LS2}, and this one is
used in the proofs of Lemmas~\ref{l:Lk} and~\ref{l:Sk}.

Assuming the validity of the lemmas we immediately obtain a proof of
Proposition~\ref{p:vn} as described just below the statement of the
proposition. Therefore, it is sufficient to prove the lemmas.

\subsection{Proofs of Lemmas~\ref{l:LS1}-\ref{l:Al}}


%
%
%
%

\begin{proof}[Proof of Lemma~\ref{l:LS1}]
  Let $A := A^L_1 = A^S_1 = \{\aa^+\}$ where $\aa^+ = 0\cdots0n$. We also assume that
  $\aa^+ \in I$
  otherwise we encounter the `empty' case. Thus \lsh-shellability of
  $I\hull{A^L_1 \cap I}$ and \ssh-shellability of
  $I\hull{A^S_1 \cap I}$ coincide with the usual shellability of
  $I\hull{A}$ (since $A$ contains a single atom). We easily observe that the
  interval $[\aa^+,\zz]$ is shellable, since it is isomorphic to $[\OO, \zz -
  \aa^+]$; and $[\OO,\zz - \aa^+]$ is shellable by our assumption. It follows that
  $I\hull{A}$ is shellable by extending every maximal chain of $[\aa^+,\zz]$ by
  $\{\OO\}$ and considering the same order of maximal chains as for shelling
  $[\aa^+,\zz]$.

\end{proof}

For the proof of a next lemma, the following claim will be useful.

\begin{claim}
\label{c:uv102}
  Let $\uu = \omega_1 \cdots\omega_{n-1}0$ be a nonzero element of $\pinch_n$
  with the last coordinate $0$, or $\uu = 1\cdots102$. Then there is $\vv \in A^S$ such that $\vv
  \prec \uu + 1\cdots102$.
\end{claim}

\begin{proof}
   If $\uu = 1\cdots102$, we can set $\vv := 1\cdots1003$, for example.

    Further, we assume $\uu \neq 1\cdots102$. Let $i$ be such that $\omega_i \geq 1$ while
    we prefer $i \neq n-1$ if possible; and furthermore, if we meet the first
    preference, we prefer $\omega_i \neq 2$ if possible. 
    
    If we meet both preferences, we set $\vv := 1\cdots121\cdots101$ where the `2'
    occurs in the $i$th position. In particular $\vv \in A^S$. We also
    have $\uu + 1\cdots102 -
    \vv = \omega_1\cdots\omega_{i-1}(\omega_i +
    1)\omega_{i+1}\cdots\omega_{n-1}1$, which is different from $\jj$ since
    $\omega_i \neq 2$. That is, $\uu + 1\cdots102 \succ \vv$.
    
    If we meet the first preference only, then we still set $\vv :=
    1\cdots121\cdots101$ where the `2' occurs on the $i$th position. This time
    we conclude $\uu + 1\cdots102 -
        \vv\neq \jj$ by realizing that there is $j \neq i,n-1$
    such that $\omega_j \neq 1$ (here we use $n \geq 4$). 
    
    Finally, it we meet no preference, then $\uu = 0\cdots0(r\cdot n)0$ for
    some integer $r$. 
    In this case, we set $\vv = 1\cdots1021$ and we
    have $\uu + 1\cdots102 - \vv = 0\cdots01(rn-2)1 \neq \jj$.

\end{proof}


\begin{proof}[Proof of Lemma~\ref{l:LS2}]
We have $\aa^L_1 = \aa^S_1 = 0\cdots0n$ and $\aa^L_2 = \aa^S_2 = 0\cdots01(n-1)$.
We set $A := \{\aa^L_1\}$ and $A^+ := \{\aa^L_1,\aa^L_2\}$. With this setting, our only task is to show that
$I\hull{A^+ \cap I}$ is \lsh-shellable (which coincide with
\ssh-shellability). We can assume that $\aa^L_2 \in I$,
otherwise $A^+ \cap I$ coincides with $A \cap I$ and we conclude by
Lemma~\ref{l:LS1}. We can also assume that $\aa^L_1 \in I$; otherwise $A^+ \cap I$
contains a single atom only and we obtain \lsh-shellability of
$I\hull{A^+ \cap I}$ in the same way as in the proof of Lemma~\ref{l:LS1}.

Altogether, we assume $\aa^L_1, \aa^L_2 \in I$ and therefore our task simplifies to
showing \lsh-shellability of $I\hull{A^+}$.
We are going to use Theorem~\ref{t:cii} for this task. For consistent notation,
we set $Q := I\hull{A^+} \setminus I\hull{A}$ and $\aa^+ = \aa^L_2$ (we prefer
using bold $\aa^+$ rather than $a^+$ in Theorem~\ref{t:cii} emphasizing that $\aa^+
\in \pinch_n$). We also recall that $I(\aa^+) = [\aa^+,\zz]$ and $\Aall(\aa^+)$ is
the set of all atoms of $I(\aa^+)$ whereas $A(\aa^+)$ is the set of only those
atoms of $I(\aa^+)$ which belong to $I\hull{A}$ as well. We need to check the
conditions of Theorem~\ref{t:cii}.

The first condition, $A$-shellability of $I\hull{A}$ just follows from
Lemma~\ref{l:LS1}.

\medskip

For checking the remaining two conditions, we need more intrinsic description
of $Q$. Note that in our notation $(\qq - \aa^+)^{(n)}$ denotes the last
coordinate of $\qq - \aa^+$. Consult Figure~\ref{f:QLS2} while following the
proof of the next claim and the rest of the proof of the lemma.

\begin{figure}
\begin{center}
  \includegraphics{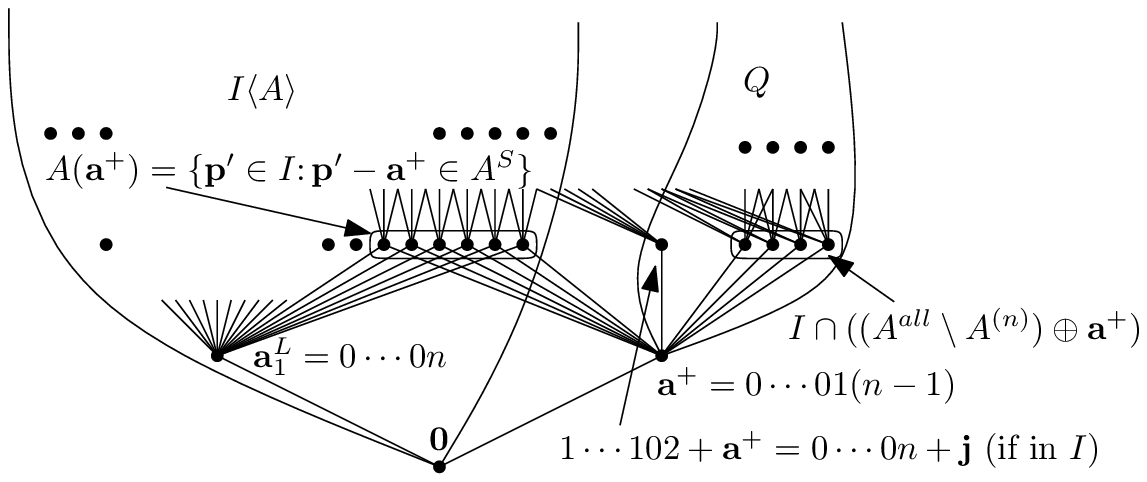}
\end{center}
\caption{Schematic drawing of $I\hull{A^+}$ in Lemma~\ref{l:LS2}.}
\label{f:QLS2}
\end{figure}

\begin{claim}
\label{c:QLS2}
  We have the following description of $Q$.
$$
Q = \{\qq \in I\hull{A^+}\colon \qq \succeq \aa^+, (\qq - \aa^+)^{(n)} = 0
\hbox{ or } \qq - \aa^+ = 1\cdots102 \}.
$$
\end{claim}

\begin{proof}
  If $\qq \in Q$, then it must satisfy $\qq \succeq \aa^+$. Therefore we can
  consider $\qq$ satisfying $\qq \succeq \aa^+$ and our task is to determine
  whether $\qq \in Q$.

  Let us first consider the case $(\qq - \aa^+)^{(n)} = 0$. Then $\qq^{(n)} =
  (\aa^+)^{(n)} = n - 1$, and therefore $\qq \not\succeq \aa^L_1 = 0\cdots0n$. We
  conclude $\qq \in Q$ since $\qq \notin I\hull{A}$.

  Now, let us consider the case $(\qq - \aa^+)^{(n)} \geq 1$. Then $\qq \geq
  \aa^L_1$. We deduce $\qq \succeq \aa^L_1$ unless $\qq = \aa^L_1 + \jj$. That is
  $\qq \notin Q$ unless $\qq = (0\cdots0n) + (1\cdots1)= 1\cdots1(n+1)$. 
  In this case $\qq - \aa^+ = 1\cdots102$.
\end{proof}

Using Claim~\ref{c:QLS2}, it is easy to check the second condition in
Theorem~\ref{t:cii}. 

We first observe that Claim~\ref{c:QLS2} implies the following description of
$A(\aa^+)$: 
\begin{equation}
\label{e:Aa+}
 A(\aa^+) = \{\pp' \in I\colon \pp' - \aa^+ \in A^S\}.
\end{equation}
Indeed, $A(\aa^+)$ consists of those elements in $I$ covering $\aa^+$ 
which do not
belong to $Q$. By Claim~\ref{c:QLS2} and the definition of $A^S$ we obtain
that $A(\aa^+)$ consists of those elements $\pp' \in I$ covering 
$\aa^+$ such that $\pp' - \aa^+ \in A^S$. This immediately yields the required
description \eqref{e:Aa+} since if $\pp' - \aa^+ \in A^S$, then $\pp' \covers
\aa^+$.


Now, by the assumptions of the lemma the interval $[\OO, \zz - \aa^+]$ is
\ssh-shellable. This interval is isomorphic to $I(\aa^+)$ by adding $\aa^+$.
Therefore, using~\eqref{e:Aa+}, this isomorphism induces a shelling of
$I(\aa^+)$ required by condition~(ii) of Theorem~\ref{t:cii}.


\medskip

Finally, we want to check condition (iii) of Theorem~\ref{t:cii}. Therefore, we
are given $\qq \in Q$ and $\pp \in I\hull{A}$ such that $\pp$ covers $\qq$. Our
task is to show that $\pp \in I(\aa^+)\hull{A(\aa^+)}$.
Recalling~\eqref{e:Aa+}, our task is to
 show that there is $\pp' \in I$ such that $\pp' - \aa^+ \in A^S$ and $\pp
 \succeq \pp'$. Note that the condition $\pp' \in I$ follows from $\pp
  \succeq \pp'$, thus we do not need to check it in the following verification
  separately.

 A natural candidate for $\pp'$ is the element $\pp'_{cand} := \aa^+ + (\pp -
 \qq)$. We have $\pp'_{cand} \preceq \pp$ since $\pp - \pp'_{cand} = \qq - \aa^+$
 and $\qq \succeq \aa^+$. Furthermore, $\pp'_{cand} - \aa^+ = \pp - \qq$; therefore
 we are done if $\pp - \qq \in A^S$. See Figure~\ref{f:LS2_third}, on the left.

 It remains to consider $\pp - \qq \notin A^S$. In this case, we have to choose
 $\pp'$ different from $\pp'_{cand}$. We further distinguish two cases 
 whether $\qq - \aa^+ = 1\cdots102$ or $(\qq - \aa^+)^{(n)} = 0$ (which is
 sufficient due to Claim~\ref{c:QLS2} using $\qq \in Q$) while we keep in mind
 that $\pp - \qq \notin A^S$. See Figure~\ref{f:LS2_third}, in the middle and
 on the right.

 \begin{figure}
\begin{center}
  \includegraphics{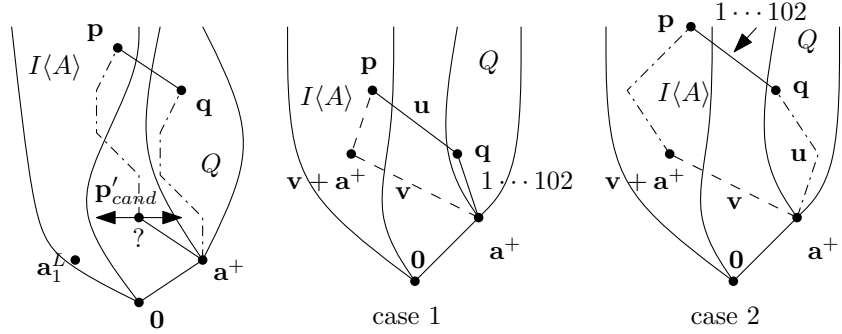}
\end{center}
\caption{Verifying condition~(iii) of Theorem~\ref{t:cii}. If $\pp - \qq$
does not belong to $A^S$ (on the left), then we need to distinguish two further cases
(in the middle and on the right). Label of an edge (a path) $\ss\ttt$ is given by
$\ttt - \ss$.}
\label{f:LS2_third}
\end{figure}

 \begin{enumerate}
   \item
    First let us assume that $\qq - \aa^+ = 1\cdots102$.


    We let $\uu := (\pp - \qq)$. In particular, either $\uu = 1\cdots102$, or 
    $\uu^{(n)} = 0$ since $\pp - \qq \notin A^S$. By
    Claim~\ref{c:uv102} there is $\vv \in A^S$ such that $\vv \prec \uu +
    1\cdots102$. Let $\pp' := \vv + \aa^+$. Then $\pp' - \aa^+ \in A^S$ and also
    $\pp' \prec \pp$ since $\pp - \pp' = (\pp - \qq) + (\qq - \aa^+) - \vv = \uu +
    1\cdots102 - \vv$ and $\uu + 1\cdots102 \succ \vv$. 
    
 \item
   Now we assume $(\qq - \aa^+)^{(n)} = 0$. Since $\pp \notin Q$,
   Claim~\ref{c:QLS2} implies that $(\pp -
   \aa^+)^{(n)} > 0$ (and $\pp - \aa^+ \neq 1\cdots102$). Therefore $(\pp -
   \qq)^{(n)} > 0$. Since $\pp - \qq \notin A^S$, we conclude $\pp - \qq =
   1\cdots102$. (This also implies that $\qq \neq \aa^+$.)

   Now let $\uu := \qq - \aa^+$. By Claim~\ref{c:uv102}, there is $\vv \in A^S$
   such that $\uu + 1\cdots102 \succ \vv$. We set $\pp' := \vv + \aa^+$. Then
   $\pp' - \aa^+ \in A^S$ and also $\pp \succ \pp'$ since $\pp - \pp' = (\pp -
   \qq) + (\qq - \aa^+) + \vv = 1\cdots102 + \uu - \vv \succ 0$. 

 \end{enumerate}

We have checked all conditions of Theorem~\ref{t:cii}. This concludes the proof
of the lemma.
\end{proof}

The following claim will be useful for the proof of the next lemma. 
Item (ii) of the claim is trivial; however, it will be useful to
refer to it as stated in the claim.
\begin{claim} \
  \label{c:a'}
  \begin{enumerate}[(i)]
    \item Let $\aa \in \Aall$ such that $\aa \neq \aa^L_1$. Then there
      is $\aa' \in \Aall$ such that $\aa' <^L \aa$ and $\aa'
      \prec \aa + \jj$. In addition, we can require $\aa' \neq 1\cdots102$.
    \item Let $\aa := \aa^L_1$. Then $\aa' \prec
      \aa + \jj$ for $\aa' = \aa^L_2$. 
  \end{enumerate}
\end{claim}

\begin{proof}
  Let us start with item (i). Let $\aa =
  0\cdots0\alpha_{\ell}\cdots\alpha_n$ where $\alpha_\ell \neq 0$. That is, we
  require $\aa' \prec 1\cdots1(\alpha_{\ell} + 1)\cdots(\alpha_n + 1)$.
  We have
  $\ell \leq n - 1$ since $\aa \neq \aa^L_1 = 0\cdots0n$. Let $\bb := 0\cdots0(\alpha_\ell -
  1)\alpha_{\ell + 1}\cdots\alpha_{n-1}(\alpha_n + 1)$. If $\bb \neq \jj,
  1\cdots102$, then
  $\bb <^L \aa$, and thus we can set $\aa' := \bb$.
  (Note that $\bb \leq \aa' + \jj$ and $\bb + \jj \neq \aa + \jj$
  implying $\bb \prec \aa + \jj$.) If $\bb = \jj$, then $\aa =
  21\cdots10$ and we can, for example, set $\aa' =
  1\cdots120$. If $\bb = 1\cdots102$, then $\aa=21\cdots101$ and we can set
  $\aa' = 1\cdots1201$.

  Item (ii) is trivial just since by definition of $\aa^L_1$ and $\aa^L_2$ we
  have $\aa^L_1 = 0\cdots0n$ and $\aa^L_2 =
  0\cdots01(n-1)$. 
\end{proof}

\begin{proof}[Proof of Lemma~\ref{l:Lk}]
  We set $A := A^L_{k-1} \cap I$ and $A^+ := A^L_k \cap I$; we also set $\aa^+ =
  \aa^L_k$. Our task is to show that if $A^+$ is nonempty, then $I\hull{A^+}$ is
  \lsh-shellable. 
  
  We can assume that $\aa^+ \in I$ otherwise \lsh-shellability of
  $I\hull{A^+}$ coincides with \lsh-shellability of $I\hull{A}$ which we
  conclude from the assumptions of the lemma (if $A^+ \neq \emptyset$).

  We can also assume that $A \neq \emptyset$, otherwise $I\hull{A^+}$
  has a single atom only and we derive the lemma analogously as
  Lemma~\ref{l:LS1}. In particular, from the assumptions of the lemma we know
  that assertion $(\As_{k-1}^L)$ is valid, and therefore we have that $I\hull{A}$ is \lsh-shellable.
  
  Our task is to use Theorem~\ref{t:ci} for verifying \lsh-shellability of
  $I\hull{A^+}$. We set $Q := I\hull{A^+}\setminus I\hull{A}$. We need to verify
  assumptions of Theorem~\ref{t:ci}.

  We have already observed that item (i) of Theorem~\ref{t:ci} is satisfied;
  that is, that $I\hull{A}$ is \lsh-shellable.

  For verifying other items, we need more intrinsic definition of $Q$. We will
  assume that $\aa^+ = 0\cdots0\alpha_{\ell}\alpha_{\ell + 1}\cdots\alpha_n$
  where $\ell$ is the smallest integer such that $\alpha_\ell > 0$. Note that
  $\ell \leq n-1$ since $k \geq 3$.

\begin{claim}
\label{c:QLk}
  We have the following description of $Q$.
  \begin{enumerate}[(i)]
    \item
$Q = \{\qq \in I\hull{A^+}\colon \qq \succeq \aa^+, (\qq - \aa^+)^{(\ell + 1)} =
  \underbrace{0\cdots0}_{n-\ell}
\}$ if $\aa^+ \neq 201\cdots1$; and
    \item
      $Q = \{\qq \in I\hull{A^+}\colon \qq \succeq \aa^+, (\qq - \aa^+)^{(2)}\in\{\underbrace{0\cdots0}_{n-1},\underbrace{10\cdots0}_{n-1} \}\}$ if $\aa^+
      =201\cdots1$.
 \end{enumerate}

\end{claim}

\begin{proof}
  First, we assume that $\aa^+ \neq 201\cdots1$ and we want to prove item~(i). Let $\qq \succeq \aa^+$. Our
  task is to determine, whether $\qq \in Q$. We also let $\qq - \aa^+ =
  \kappa_1\cdots\kappa_n$.

  We need to show two inclusions. 
  \begin{itemize}
\item
  For the first one, we assume that  $(\qq-\aa^+)^{(\ell + 1)} \neq
  \underbrace{0\cdots0}_{n-\ell}$, and
  we want to show that $\qq \notin Q$. That is, we want to find an atom from
  $A$ which is below $\qq$.  In this case, we have 
  $i \in \{\ell+1, \dots, n\}$ such that $\kappa_i \neq 0$. Let 
  $$\aa :=
  0\cdots0(\alpha_{\ell} - 1)\alpha_{\ell + 1}\cdots\alpha_{i-1}(\alpha_i +
  1)\alpha_{i+1}\cdots\alpha_n.$$
  We have $0 < \aa < \qq$. 
  
  If $\aa \neq \jj$ and $\qq -
  \aa \neq \jj$, then $0 \prec \aa \prec \qq$, and thus $\aa$ is
  the required atom of $A$ since $\aa$ precedes $\aa^+$ in the $<^L$
  order.

  If $\aa = \jj$, then $\aa^+ = 21\cdots101\cdots1$ where the `0' appears in the
  $i$th position ($i \geq 3$ since $\aa^+ \neq 201\cdots1$). In particular, if
  $\aa' = 201\cdots1$, then $\qq \geq \aa'$
  (since $\kappa_i \geq 1$) and $\aa'$ precedes $\aa^+$ in the $<^L$ order. 
  Therefore, $\aa'$ is the required atom of $A$ unless $\qq = \aa' +
  \jj = 312\cdots2$. In this case, we can use $1\cdots102$ for example. 
  
  If $\qq - \aa = \jj$, and $\aa \neq \jj$ we consider $\aa' \prec \aa + \jj = \qq$ obtained from
  Claim~\ref{c:a'}. We also have $\aa' <^L \aa^+$. This follows from
  Claim~\ref{c:a'}~(i) by $\aa' <^L \aa <^L \aa^+$ if $\aa \neq \aa^L_1$. 
  It follows from Claim~\ref{c:a'}~(ii) if $\aa = \aa^L_1$ since
  $\aa' <^L \aa^L_3 \leq^L \aa^+$. 
  

\item
  For the second inclusion, we assume that  $(\qq-\aa^+)^{(\ell + 1)} =
  \underbrace{0\cdots0}_{n-\ell}$, and we need to show that $\qq \in Q$; that
  is, we need to show that $\aa \not\preceq \qq$ for any $\aa \in A$.

  Let $\aa = \alpha'_1\cdots\alpha'_n \in A$. Since $\aa <^L \aa^+$, we have
  that $\alpha'_1 = \cdots = \alpha'_{\ell - 1} = 0$ and $\alpha'_{\ell} \leq
  \alpha_{\ell}$. This implies that there is $i \in \{\ell + 1, \dots, n\}$
  such that $\alpha'_i > \alpha_i$ (note that $\alpha_1 + \cdots + \alpha_n = n
  = \alpha'_1 + \cdots + \alpha'_n$ since both $\aa^+$ and $\aa$ are atoms;
  note also that we get a strict inequality since $\aa^+ \neq \aa$). This implies $\qq \not\succeq \aa$ since
  $\qq$ and $\aa^+$ agree in the $i$th position.
 \end{itemize}


Now, we want to prove item~(ii). That is, we assume that $\aa^+ = 201\cdots1$.
Similarly as before, let $\qq \succeq \aa^+$. Our
  task is to determine, whether $\qq \in Q$. We also let $\qq - \aa^+ =
  \kappa_1\cdots\kappa_n$. We again need to show two inclusions.

\begin{itemize}
\item
  For the first one, we assume that  $(\qq-\aa^+)^{(2)} \notin 
  \{\underbrace{0\cdots0}_{n-1}, \underbrace{10\cdots0}_{n-1}\}$, and we want to show that $\qq \notin Q$.

  If we and apply the reasoning from item~(i), we obtain that 
  $\qq \notin Q$ if $\kappa_i > 0$ for some $i \geq 3$.

  It remains to consider the case $(\qq-\aa^+)^{(2)} = \kappa_20\cdots0$ where
  $\kappa_2 \geq 2$. In this case, we set $\aa = 021\cdots1$. Thus $\qq >
  \aa$. In addition, $\qq \neq \aa + \jj$ since $(\qq-\aa^+)^{(2)} =
  \kappa_20\cdots0$. Thus, $\qq \succ \aa$. We also have $\aa <^L \aa^+$, and
  therefore $\qq \notin Q$.
\item
  For the second inclusion, we assume that $(\qq-\aa^+)^{(2)} \in 
    \{\underbrace{0\cdots0}_{n-1}, \underbrace{10\cdots0}_{n-1}\}$ and we need
    to show that $\aa \not\preceq \qq$ for any $\aa \in A$.

  Let $\aa = \alpha'_1\cdots\alpha'_n \in A$. Since $\aa <^L \aa^+$, we have
  that $\alpha'_1 \leq 2$. This implies that 
  either is $i \in \{3, \dots, n\}$
  such that $\alpha'_i > \alpha_i = 1$, 
  or $\alpha'_2 > \alpha_2 = 0$ and $\alpha'_i = \alpha_i = 1$ for $i \geq 3$.
  
  In the first case we have $\qq \not\succeq \aa$ since
  $\qq$ and $\aa^+$ agree in the $i$th position. In the second case, we have
  $\alpha_2 \geq 2$ since $\aa \neq \jj$. Therefore, again $\qq \not\succeq
  \aa$, since $\qq$ exceeds $\aa^+$ in the second position at most by $1$. 
\end{itemize}

\end{proof}

\medskip

Now we verify condition~(ii) of Theorem~\ref{t:ci}. Let $J = [\aa^+,\qq]$ be an
interval where $\qq \in Q$. We recall that $[\aa^+,\qq]$ is isomorphic to
$[\OO,\qq - \aa^+]$.

If $\aa^+ \neq 201\cdots1$, then by Claim~\ref{c:QLk}, $J$ is isomorphic to an
interval in $\ver_{n,\ell}$ (by forgetting last $n - \ell$ coordinates of $J
\ominus \aa^+$). Therefore, $J$ is shellable by Proposition~\ref{p:non_pinch}.

If $\aa^+ = 201\cdots1$, then $Q$ has a very simple structure by
Claim~\ref{c:QLk}; see Figure~\ref{f:Q(ii)}. We could check that every interval
in $Q$ in this case is a modular lattice and deduce shellability of $Q$ in the
same way as in Proposition~\ref{p:non_pinch}, using~\cite[Theorem
3.7]{bjorner80}. However, this is perhaps just an overkill in this case and
the shelling order of every interval can be easily found explicitly.

 \begin{figure}
\begin{center}
  \includegraphics{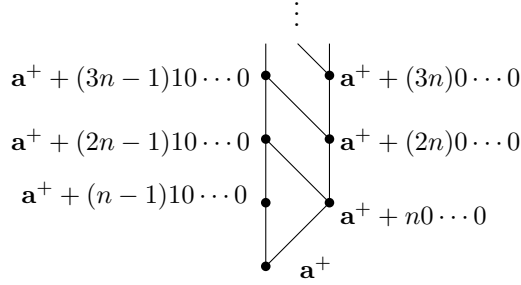}
\end{center}
\caption{The structure of $Q$ in item~(ii) of Claim~\ref{c:QLk} (or
Claim~\ref{c:QSk}).}
\label{f:Q(ii)}
\end{figure}

\medskip

We continue with the verification of condition~(iii) of Theorem~\ref{t:ci}; that is,
we verify the edge-falling property. Let $\qq \in Q$, $\qq' \in Q \cup \{\OO\}$ and $\pp \in
I\hull{A}$ be such that $\pp \covers \qq$ and $\qq \covers \qq'$. Our task is
to find $\pp' \in I\hull{A}$ such that $\pp \covers \pp' \covers \qq'$.

Natural candidate for $\pp'$ is $\pp'_{cand} := \qq' + (\pp - \qq)$. 
We have $\pp \covers \pp'_{cand} \covers \qq'$.
If $\aa^+ \neq 201\cdots1$, we immediately obtain that $\pp'_{cand} \in I\hull{A}$ from
Claim~\ref{c:QLk}~(i) as follows. We know that $(\qq - \qq')^{(\ell + 1)} =
\underbrace{0\cdots0}_{n-\ell}$ by Claim~\ref{c:QLk}~(i) since $\qq - \qq' =
(\qq - \aa^+) - (\qq' - \aa^+)$. Therefore $\pp^{(\ell + 1)} =
(\pp'_{cand})^{(\ell + 1)}$, and it follows by Claim~\ref{c:QLk}~(i) that
$\pp'_{cand}$ indeed belongs to $I\hull{A}$. Therefore we can set $\pp' := \pp'_{cand}$.

If $\aa^+ = 201\cdots1$, we need to be more careful. We have $\pp^{(2)} -
(\pp'_{cand})^{(2)} = \qq^{(2)} - (\qq')^{(2)}$. Therefore, if
$\qq^{(2)} =  (\qq')^{(2)}$, then we obtain $\pp'_{cand} \in I\hull{A}$
by Claim~\ref{c:QLk}~(ii) and we can set $\pp' := \pp'_{cand}$. However, it
might also occur that $(\qq - \aa^+)^{(2)} = 10\cdots0$ and $(\qq' - \aa^+)^{(2)} =
0\cdots0$ by Claim~\ref{c:QLk}~(ii). In this case, we focus on $(\pp -
\qq)^{(2)}$. Claim~\ref{c:QLk}~(ii) implies that $(\pp -
\qq)^{(2)} \neq 0\cdots0$. If $(\pp -
\qq)^{(2)} \neq 10\cdots0$, then $\pp'_{cand} \in I\hull{A}$ again by
Claim~\ref{c:QLk}~(ii) and we can again set $\pp' := \pp'_{cand}$. 

Finally, it remains to consider the case $(\pp - \qq)^{(2)} = 10\cdots0$. In
this case $\pp'_{cand} \in Q$ and we have to choose $\pp'$ differently. We
actually obtain $\pp - \qq = (n-1)10\cdots0$ since $\pp \covers \qq$.
Similarly, we obtain $\qq - \qq' = (n-1)10\cdots0$. We can then choose $\pp' :=
\qq' + (n-2)20\cdots0$. Then $\pp \covers \pp' \covers \qq'$ and $\pp' \in
I\hull{A}$ by Claim~\ref{c:QLk}. See Figure~\ref{f:ver_ef}.
\begin{figure}
\begin{center}
  \includegraphics{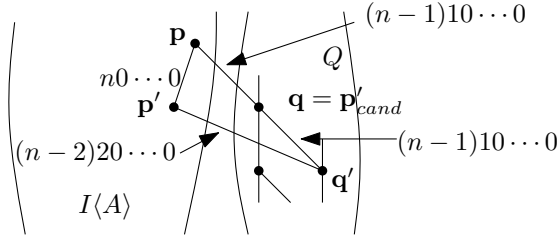}
\end{center}
\caption{The last case of the verification of the edge-falling property. Similarly
as in Figure~\ref{f:LS2_third}, the label of an edge $\ss\ttt$ is $\ttt - \ss$.}
\label{f:ver_ef}
\end{figure}

\medskip

We conclude by verifying condition~(iv) of Theorem~\ref{t:ci}. Let $\qq \in Q$,
we need to show that the poset $I(\qq)\hull{A(\qq)}$ is shellable where
$A(\qq)$ is defined as in the statement of the theorem. We observe
that this poset is isomorphic with $I(\qq)\hull{A(\qq)} \ominus \qq$, that is,
with $[\OO,\zz-\qq]\hull{A(\qq) \ominus \qq}$. Note that $\rk(\zz - \qq) <
\rk(\zz)$. Here we plan to use our assumption that Proposition~\ref{p:vn} is
valid for intervals $[\OO,\yy]$ with $\rk(\yy) < \rk(\zz)$, in particular, for
the interval $[\OO,\zz-\qq]$. Therefore, we want to determine $A(\qq) \ominus \qq$.

Let $\aa \in \Aall$, we want to determine, whether $\aa \in A(\qq) \ominus
\qq$. This is equivalent with determining whether $\qq + \aa \in A(\qq)$
and using the definition of $A(\qq)$ with determining whether $\qq + \aa \in 
I\hull{A}$
(assuming that $\qq + \aa \in I(\qq)$, otherwise $\aa \notin A(\qq)
\ominus \qq$).

If $\aa^+ \neq 201\cdots1$, we get that $\qq + \aa \in I\hull{A}$ if and only if
$\aa \in A^{(\ell + 1)}$ and $\qq + \aa \in I(\qq)$ by Claim~\ref{c:QLk}~(i). 
Therefore, we obtain the required shellability of $[\OO,\zz-\qq]\hull{A(\qq)
\ominus \qq}$ from assertion $(\As^{(\ell + 1)}_k)$
(with $k = |\A^{(\ell + 1)}|$) for the interval $[\OO,\zz-\qq]$.

If $\aa^+ = 201\cdots1$ and $(\qq - \aa^+)^{(2)} = 1\cdots0$, then $\qq + \aa^+ \in
I\hull{A}$ if and only if $\aa \in A^{(2)}$ and $\qq + \aa \in I(\qq)$ by
Claim~\ref{c:QLk}~(ii).  Therefore, we obtain the required shellability of
$[\OO,\zz-\qq]\hull{A(\qq) \ominus \qq}$ from assertion $(\As^{(2)}_k)$
(with $k = |\A^{(2)}|$) for the interval $[\OO,\zz-\qq]$.

If $\aa^+ = 201\cdots1$ and $(\qq - \aa^+)^{(2)} = 0\cdots0$, then $\qq + \aa^+ \in
I\hull{A}$ if and only if $\aa \in \Aall \setminus\{(n-1)10\cdots0,
n0\cdots0\}$ and $\qq + \aa \in I(\qq)$ by Claim~\ref{c:QLk}~(ii). Luckily,
$\Aall \setminus\{(n-1)10\cdots0, n0\cdots0\}$ is $\Aall$ minus the latest two
elements of $\Aall$ in the $<^L$ order. Therefore, we obtain the required
shellability of $[\OO,\zz-\qq]\hull{A(\qq) \ominus \qq}$ 
from assertion $(\As^L_k)$ (with $k = |\Aall| - 2$) for the interval
$[\OO,\zz-\qq]$.

This covers all cases when $\aa^+ = 201\cdots1$ by Claim~\ref{c:QLk}~(ii). Thus,
we have verified condition~(iv) of Theorem~\ref{t:ci} which concludes the proof
of the lemma.
\end{proof}

For the proof of the next lemma, we need the following extension of
Claim~\ref{c:a'}.

\begin{claim} \
  \label{c:a'l}
Let $\ell \in [n-1]$. Let $\aa \in \Aall$ such that $\aa \neq \aa^L_1$. Then there
is $\aa' \in A^{(\ell+1)}$ such that $\aa' <^L \aa$ and $\aa'
      \prec \aa + \jj$. In addition, we can assume $\aa' \neq 1\cdots102$.
\end{claim}

\begin{proof}
 By Claim~\ref{c:a'}~(i) we have $\bb'' \in \Aall$ (playing the role of $\aa'$
 in Claim~\ref{c:a'}) such that $\bb'' <^L \aa$
 and $\bb'' \prec \aa + \jj$ and $\bb'' \neq 1\cdots102$. If $\bb'' \in
 A^{(\ell+1)}$, then we set $\aa'
 := \bb''$ and we are done.

 If $\bb'' \notin A^{(\ell+1)}$, then $\bb'' := \beta_1\cdots\beta_{n-1}0$ for
 some $\beta_1,\dots,\beta_{n-1} \geq 0$. Let $i \in [n-1]$ be such that
 $\beta_i \neq 0$ and $\beta_i$ is as small as possible. We set $\aa' :=
 \beta_1\cdots\beta_{i-1}(\beta_i - 1)\beta_{i+1}\cdots\beta_{n-1}1$. We have
 that $\aa' \neq \jj$ due to our choice that $\beta_i$ is as small as
 possible. Thus $\aa' <^L \bb'' <^L \aa$. In addition $\aa' \prec \aa + \jj$
 since $\aa' \leq \aa + \jj$ ($\aa'$ is dominated by $\bb''$ in the first
 $n-1$  coordinates and dominated by $\jj$ in the last coordinate) and $\aa'
 \neq \aa$. Finally, $\aa' \in A^{(\ell+1)}$ and $\aa' \neq 1\cdots102$ since
 its last coordinate is $1$.
\end{proof}

\begin{proof}[Proof of Lemma~\ref{l:Sk}]
  The proof is similar to the proof of Lemma~\ref{l:Lk}. It is only slightly
  more technical, since the $<^S$ order is more complicated than the $<^L$
  order.
  
  We set $A := A^S_{k-1} \cap I$ and $A^+ := A^S_k \cap I$; we also set $\aa^+ =
  \aa^S_k$. Our task is to show that if $A^+$ is nonempty then $I\hull{A^+}$ is
  \ssh-shellable. 
  
  Similarly as in the proof of Lemma~\ref{l:Lk}, we derive that we can assume
  $\aa^+ \in I$, $A \cap I \neq \emptyset$ and therefore $I\hull{A}$ is
  \ssh-shellable from the assumptions of this lemma.

  Our task is to use Theorem~\ref{t:ci} for verifying \ssh-shellability of
  $I\hull{A^+}$. We set $Q := I\hull{A^+}\setminus I\hull{A}$. We need to verify
  assumptions of Theorem~\ref{t:ci}.

  We have already observed that item (i) of Theorem~\ref{t:ci} is satisfied;
  that is, that $I\hull{A}$ is \ssh-shellable.

  For verifying other items, we need more intrinsic definition of $Q$. We will
  assume that $\aa^+ = 0\cdots0\alpha_{\ell}\alpha_{\ell + 1}\cdots\alpha_n$
  where $\ell$ is the smallest integer such that $\alpha_\ell > 0$. Note that
  $\ell \leq n-1$ since $k \geq 3$.

\begin{claim}
\label{c:QSk}
  We have the following description of $Q$.
  \begin{enumerate}[(i)]
    \item
$Q = \{\qq \in I\hull{A^+}\colon \qq \succeq \aa^+, (\qq - \aa^+)^{(\ell+1)} =
  \underbrace{0\cdots0}_{n-\ell}
\}$ if $\aa^+ \neq 1\cdots102,\allowbreak 201\cdots1,\allowbreak 201\cdots102$;
    \item
      $Q = \{\qq \in I\hull{A^+}\colon \qq \succeq \aa^+, (\qq - \aa^+)^{(2)}\in\{\underbrace{0\cdots0}_{n-1},\underbrace{10\cdots0}_{n-1} \}\}$ if $\aa^+
      =201\cdots1$ or $\aa^+ = 201\cdots102$; and
    \item $Q = \{\aa^+\}$ if $\aa^+ = 1\cdots102$.
 \end{enumerate}
\end{claim}

 Note that we crucially use that $n \geq 4$ in order that this claim makes
 sense; that is, we use that $20\underbrace{1\cdots1}_{n-4}02$ belongs to
 $\pinch_n$.

\begin{proof}
  The proof is similar to the proof of Claim~\ref{c:QLk}; however, in this
  proof there are more cases to consider. Keeping in mind the
  number of cases we want to consider, we use slightly different approach how
  to treat them, compared to Claim~\ref{c:QLk}.

  We assume that we are given $\qq$ such that $\qq \succeq \aa^+$ (this is a
  necessary condition for $\qq \in Q$).
  We let $\qq - \aa^+ =
  \kappa_1\cdots\kappa_n$. If $\aa^+ \notin \{1\cdots102,\allowbreak
    201\cdots1,\allowbreak 
  201\cdots102\}$ we want to
  verify that $\qq \in Q$ if and only if $\kappa_{\ell+1} = \cdots = \kappa_n =
  0$. If $\aa^+ \in \{1\cdots102, 201\cdots102\}$ we want to
  verify that $\qq \in Q$ if and only if $\kappa_2 \in \{0,1\}$ and $\kappa_3
  = \cdots = \kappa_n = 0$. If $\aa^+ = 1\cdots102$, we want to verify that $\qq
  \in Q$ if and only if $\qq = \aa^+$.

  First, we distinguish cases according to whether
  $\kappa_{\ell+1}\cdots\kappa_{n} = 0\cdots0$ (note that we also cover $\aa^+
  \in \{1\cdots102, 201\cdots1, 201\cdots102\}$ by setting $\ell = 1$ in these cases).

  \begin{enumerate}
    \item $\kappa_{\ell+1}\cdots\kappa_{n} \neq 0\cdots0$.
      In this case we have $i \in \{\ell + 1, \cdots, n\}$ such that $\kappa_i
      > 0$. We prefer $i \neq 2$, if possible. We set
      $$
       \aa := 0\cdots0(\alpha_{\ell} - 1)\alpha_{\ell +
       1}\cdots\alpha_{i-1}(\alpha_i + 1)\alpha_{i+1}\cdots\alpha_n.
      $$
     Note that if $\aa \neq \jj$, then $\aa$ precedes $\aa^+$ in the $<^L$
     order. (In fact, $\aa$ precedes $\aa^+$ in the lexicographic order in any
     case, but we do not define the $<^L$ order for $\jj$.) Note also that $\aa^+
     < \qq$. In some cases, we will manage to show that
     $\aa \neq \jj$, $\aa <^S \aa^+$ and $\aa + \jj \neq \qq$. This will imply
     that $\aa \in A$ and $\aa \prec \qq$ and therefore $\qq \not \in
     Q$. In some other cases we will replace $\aa$ with another $\aa'$
     satisfying the above-mentioned conditions still deriving $\qq \notin Q$.
     However, this will be impossible if $\aa^+ \in \{1\cdots102,
     201\cdots102\}$, $i = 2$ and $\kappa_2 = 1$ when we will actually derive
     that $\qq \in Q$.

     Now we distinguish several subcases according to $\aa^+$.

     \begin{enumerate}
       \item
	 $\aa^+ \in A^{(n)} = A^S \cup \{1\cdots102\}$.

	 Before we start, we remark that all considerations are also valid if
	 $\aa^+ = 1\cdots102$. The atom $1\cdots102$ is the last atom of
	 $A^{(n)}$ in the $<^S$ order. This will reflect in such a way, that in
	 some cases we check for $\aa^+ = 1\cdots102$ more than we need (which is
	 not a big price for a coherent case analysis). 
	 
	 We have that $\aa$ precedes $\aa^+$ in the $<^S$ order unless $\aa \in
	 \{\jj,\allowbreak 1\cdots102\}$. Therefore, for the beginning we assume that
	 $\aa \notin \{\jj, 1\cdots102\}$. If, in addition, $\aa + \jj \neq
	 \qq$, then we have the required properties of $\aa$ deriving $\qq
	 \notin Q$. However, if $\aa + \jj = \qq$, then we obtain $\aa' \neq
	 1\cdots102$ of required properties by Claim~\ref{c:a'l} (or by
	 Claim~\ref{c:a'}~(ii) if $\aa = 0\cdots0n$).

	 If $\aa = \jj$, then $\aa^+ = 21\cdots101\cdots1$ where the `0' appears
	 in the $i$th position. 
	 
	 We distinguish subsubcases according to $i$.
	 
	 \begin{enumerate}
	   \item
	     $i \geq 3$.
	     
	     In this situation we set $\aa' = 201\cdots1$. Then $\qq > \aa'$
	     (since $\kappa_i \geq 1$ and $\qq \succ \aa^+$) and $\aa'$ precedes
	     $\aa^+$ in the $<^L$ order and therefore in $<^S$ order as well. Therefore, $\aa'$ has the required
	     properties unless $\qq = \aa' + \jj = 312\cdots2$. 
	     In this case, we can use $1\cdots1201 \prec \qq$, for example.

	   \item
	     $i = 2$.

	     In this situation $\aa^+ = 201\cdots1$. We also have $\kappa_3 =
	     \cdots = \kappa_n = 0$ since we wanted $i \neq 2$ if possible.
	     
	     If $\kappa_2 \geq 2$, implying $\qq \geq 221\cdots1$,
	     we still can set $\aa' = 021\cdots1$ deriving $\qq \notin Q$ (note
	     that $\qq \neq \aa + \jj$ since $\kappa_n = 0$).

	     If $\kappa_2 = 1$, we actually want to derive $\qq \in Q$
	     according to our description. In this case, it is easiest to refer
	     to Claim~\ref{c:QLk}~(ii) (since we have already done this
	     analysis). The claim implies that there is no $\aa \in \Aall$
	     such that $\aa <^L \aa^+$ and $\aa \preceq \qq$. In particular,
	     there is no such $\aa \in A^{(n)}$. Since $\aa <^L \aa^+$ is
	     equivalent with $\aa <^S \aa^+$ in this case, we deduce $\qq \in Q$.

         \end{enumerate}
	 
	 If $\aa = 1\cdots102$, then we can perform the same analysis as if
	 $\aa = \jj$ just replacing the suffix $111$ with $102$. (The only
	 major difference is that we cannot use the shortcut referring to
	 Claim~\ref{c:QLk}.) Here the analysis follows in detail.

	 We have $\aa^+ = 21\cdots101\cdots102$ where the first `0' appears
	 in the $i$th position or $\aa^+ = 21\cdots101$ if $i = n$. (In particular
	 $i \neq n-1$.)
	 
	 We distinguish subsubcases according to $i$.
	 
	 \begin{enumerate}
	   \item
	     $i \geq 3$.
	     
	     In this situation we set $\aa' = 201\cdots102$. Then $\qq > \aa'$
	     (since $\kappa_i \geq 1$ and $\qq \succ \aa^+$) and $\aa'$ precedes
	     $\aa^+$ in the $<^L$ order (hence in $<^S$ order as well). 
	     Therefore, $\aa'$ has the required
	     properties unless $\qq = \aa' + \jj = 312\cdots213$. 
	     In this case, we can use $1\cdots1201 \prec \qq$, for example.

	   \item
	     $i = 2$.

	     In this situation $\aa^+ = 201\cdots102$. We also have $\kappa_3 =
	     \cdots = \kappa_n = 0$ since we wanted $i \neq 2$ if possible.
	     
	     If $\kappa_2 \geq 2$, implying $\qq \geq 221\cdots102$,
	     we still can set $\aa' = 021\cdots102$ deriving $\qq \notin Q$ (note
	     that $\qq \neq \aa + \jj$ since $\kappa_n = 0$).

	     If $\kappa_2 = 1$, we actually want to derive $\qq \in Q$
	     according to our description. In this case $\qq = (r\cdot n +
	     1)1\cdots102$ for some positive integer $r$. 
	     We want to show that there is no $\aa \in A^S$
	     such that $\aa <^S \aa^+$  and $\aa \prec \qq$. For contradiction,
	     there is such $\aa = \alpha'_1\cdots\alpha'_n$. Condition $\aa
	     <^S \aa^+$ implies $\alpha'_1 \leq 2$.
	     Since the sum of the last $(n-1)$ coordinates of $\qq$ equals $n - 1$, 
	     we derive either that $\alpha'_1 = 1$ and $\aa$ agrees with
	     $\qq$ on all remaining $n-1$ coordinates or that $\alpha'_1 = 2$
	     and $\aa$ agrees with $\qq$ on all remaining $n-1$ coordinates
	     except one coordinate, where it is one less. The first case is
	     excluded since $1\cdots102 \not<^S \aa^+$. The second case is also
	     excluded, since in such a case $\aa \not<^L \aa^+$, implying $\aa
	     \not<^S \aa^+$. A contradiction. We conclude that $\qq \in Q$ if
	     $\kappa_2 = 1$.
	     
	   \end{enumerate} 

       \item
	 $\aa^+ \in \Aall \setminus A^{(n)}$.

	 In this case, $\alpha_n = 0$. We also emphasize that $\aa$ precedes
	 $\aa^+$ in the $<^S$ order if $\aa \neq \jj$. This is simply 
	 because, $\aa$ precedes $\aa^+$ in the $<^L$ order and $\aa^+ \notin
	 A^{(n)}$ in this case. Therefore we derive $\qq \notin Q$ if $\aa
	 \neq \jj$ and $\aa + \jj \neq \qq$.

	 Now, let us consider the case $\aa \neq \jj$, but $\aa + \jj = \qq$.
	 Then there is $\aa'$ from Claim~\ref{c:a'}~(i) or~(ii) such that 
	 $\aa' <^L \aa^+$, $\aa \prec \qq$. We derive $\aa' <^S \aa^+$, and
	 therefore $\qq \notin Q$.

	 Finally, we consider the case $\aa = \jj$. We derive $\aa^+ =
	 21\cdots10$ and $i = n$ (since $\alpha_n = 0$). We set $\aa' :=
	 21\cdots101$ or $\aa' := 21\cdots1011$ so that $\aa' + \jj \neq
	 \qq$. We derive $\aa' \prec \qq$, $\aa' <^S \aa^+$, and therefore $\qq
	 \notin Q$ as desired.

     \end{enumerate}

    \item $\kappa_{\ell+1}\cdots\kappa_{n} = 0\cdots0$.

    In this case we want to derive $\qq \in Q$ for 
    all possible choices of $\aa^+$ except $\aa^+ = 1\cdots102$ and $\qq \succ
    \aa^+$.
 
    We distinguish subcases according to $\aa^+$.

    \begin{enumerate}
      \item $\aa^+ \in A^S$. 
	
	In this case we refer to Claim~\ref{c:QLk} which
	implies that there is no $\aa \in \Aall$ such that $\aa <^L \aa^+$ and
	$\aa \prec \qq$. Therefore, in particular, there is no $\aa \in
	A^S$ with $\aa <^S \aa^+$ and $\aa \prec \qq$ which is what we need.
      \item $\aa^+ = 1\cdots102$.

	If $\qq = \aa^+$, then $\qq \in Q$ as desired.

	If $\qq \succ \aa^+$, then $\qq = (rn+1)1\cdots102$ for some integer $n$.
	Setting $\aa = 21\cdots101$ we get $\aa \in A^{(n)}$ implying 
	$\aa <^S \aa^+$ and also $\aa \prec \qq$. Thus $\qq \notin Q$ as
	required.
     \item $\aa^+ \in \Aall \setminus A^{(n)}$.

       By Claim~\ref{c:QLk} there is no $\aa \in \Aall$ such that $\aa <^L
       \aa^+$ and $\aa \prec \qq$. Therefore, in particular, there is no $\aa
	\in \Aall \setminus A^{(n)}$ with $\aa <^S \aa^+$ and $\aa \prec \qq$.

      On the other hand, there is no $\aa \in A^{(n)}$ with $\aa <^S \aa^+$ and
      $\aa \prec \qq$ either, because $\alpha_n = \kappa_n = 0$ implying that
      the last coordinate of $\qq$ is $0$ whereas $\aa$ from $A^{(n)}$ has the
      last coordinate nonzero.

      Altogether, there is no $\aa \in \Aall$ with $\aa <^S \aa^+$ and $\aa
      \prec \qq$ implying $\qq \in Q$.

    \end{enumerate}

  \end{enumerate}
This finishes the proof of the claim.
\end{proof}

Now we verify condition~(ii) of Theorem~\ref{t:ci}. However, the verification
is almost the same as in case of Lemma~\ref{l:Lk} using Claim~\ref{c:QSk}
instead of Claim~\ref{c:QLk}. This is because of the described structure of $Q$.
(Compare with the text below the proof of Claim~\ref{c:QLk}.)

If $\aa \notin\{1\cdots102, 201\cdots1, 201\cdots102\}$, then we just use
Proposition~\ref{p:non_pinch}. If $\aa \in \{201\cdots1, 201\cdots102\}$, then
we obtain shellability of $Q$ referring to Figure~\ref{f:Q(ii)}. Finally, if
$\aa = 1\cdots102$ then the verification is trivial, since a poset with single
element is shellable.

\medskip

We continue with the verification of condition~(iii) of Theorem~\ref{t:ci}; that is,
we verify the edge-falling property. If $\aa^+ \neq 1\cdots102$, then again 
this verification can be taken in verbatim
from the analogous verification in the proof of Lemma~\ref{l:Lk} using
Claim~\ref{c:QSk} instead of Claim~\ref{c:QLk}, considering cases according to
structure of $Q$. We therefore do not repeat the relevant text
again.

If $\aa^+ = 1\cdots102$, then the verification of the edge falling property is
somewhat trivial. In this case $Q = \{\aa^+\}$ by Claim~\ref{c:QSk}. Therefore,
we are supposed to verify that if $\qq = \aa^+$, $\qq' = \OO$, and $\pp \in
I\hull{A}$ is such that $\pp \covers \qq$, then there is $\pp' \in A$ covering $\OO$ and
covered by $\pp$. But this just immediately follows from $\pp \in
I\hull{A}$ since $\rk(\pp) = 2$.

\medskip

We conclude by verifying condition~(iv) of Theorem~\ref{t:ci}. We again refer
that if $\aa^+ \neq 1\cdots102$, then this verification is already done in the
proof of Lemma~\ref{l:Lk}. It again solely depends on the structure of $Q$.

If $\aa^+ = 1\cdots102$, then we are just supposed to check that the interval
$[\aa^+,\zz]$ is shellable. This follows from the assumptions of this lemma,
since it is isomorphic to $[\OO, \zz - \aa^+]$.
  
\end{proof}

\begin{proof}[Proof of Lemma~\ref{l:Al}]
  First we observe that it is sufficient to prove the lemma for case $k =
  |\Aall|$ since an \lsh-shelling of $I\hull{A_{j+1}^{(\ell + 1)} \cap I}$
  restricts to an \lsh-shelling of $I\hull{A_{j}^{(\ell + 1)} \cap I}$.
  Therefore, in case $k = |\Aall|$ we just aim to show that $I\hull{\Aall \cap
  I}$ is \lsh-shellable.
  
  We plan to use Theorem~\ref{t:ciii} for the proof of this lemma where we set
  $A := \Aall \cap I$ and $A' := A^{(\ell + 1)} \cap I$.

  Condition~(i) of Theorem~\ref{t:ciii} follows from the assumptions of
  the lemma.

  For checking condition~(ii), we have $\bb \in I \cap (\Aall \setminus
  A^{(\ell+1)})$ and $\pp \in I\hull{A^{(\ell+1)} \cap I}$ covering $\bb$. We need
  to find $\bb' \in A^{(\ell+1)} \cap I$ such that $\pp \covers \bb'$ and $\bb'
  <^L \bb$. Actually, we will only check $\bb' \in A^{(\ell+1)}$, $\pp \covers
  \bb'$, and $\bb' <^L \bb$ since $\pp \covers \bb'$ implies $\bb' \in I$.

  We have that $\bb = \beta_1\cdots\beta_{\ell}0\cdots0$ since $\bb \notin
  A^{(\ell+1)}$. On the other hand, if we let $\pp = \pi_1\cdots\pi_n$, then there
  is $j \in \{\ell+1,\dots,n\}$ such that $\pi_j > 0$ since $\pp \in
  I\hull{A^{(\ell+1)} \cap I}$. Let also $i \in [\ell]$ be such an index that
  $\beta_i > 0$ and $\beta_i$ is as small as possible. We set the following
  candidate for $\bb'$.
  $$
  \bb'_{cand} := \beta_1\cdots\beta_{i-1}(\beta_i -
  1)\beta_{i+1}\cdots\beta_{\ell}0\cdots010\cdots0,
  $$
  where the `1' appears on the $j$th position. We have $\bb'_{cand} \leq \pp$.
  We also have $\bb'_{cand} \neq \jj$; this is obvious if $\ell \neq n-1$ and
  it follows from our choice of $\beta_i$ if $\ell = n-1$. In particular
  $\bb'_{cand} <^L \bb$ and $\bb'_{cand} \in A^{(\ell+1)}$. 
  If $\bb'_{cand} + \jj \neq \pp$, then $\bb'_{cand} \prec \pp$ and
  consequently $\pp \covers \bb'_{cand}$ (by comparing ranks). 
  Thus, we can simply set $\bb' := \bb'_{cand}$ in this case.

  If $\bb'_{cand} + \jj = \pp$, we think of $\bb'_{cand}$ as $\aa$ from
  Claim~\ref{c:a'l}. We obtain the corresponding $\aa'$ and we just set $\bb'
  := \aa'$.

\end{proof}

\section{Relation of lexicographic shellability and $A$-shell\-ability.}
\label{s:lex}
\subsection{Lexicographic shellability}

Here we briefly recall the definition of lexicographic shellability. The reader
interested in more details (including examples) is referred to sources such
as~\cite{bjorner-wachs82, bjorner-wachs83, kozlov97, kozlov08}. The reader familiar with lexicographic shellability can skip this subsection.

As usual, we let $(P, \leq)$ be a graded poset (with a unique minimal and
maximal
element), using the notation from Section~\ref{s:crit}.
Given a maximal chain $c \in C(P)$ we label all of its edges with elements of
some poset $\Lambda$ (typically, $\Lambda = \Z$). In this way we label edges of
all maximal chains in $C(P)$ (that is, a label of an edge might differ if we
start with two different chains). We obtain a \emph{chain-edge labeling} of
$P$ if the following condition is satisfied. Whenever $c, c' \in C(P)$ are two
chains sharing first $k$ edges (for some $k$), then the labels of these first
$k$ edges have to coincide. Let us assume that $P$ is equipped with a fixed
chain-edge labeling.

A \emph{rooted interval} $[x,y]_r$ is an interval in $P$ where the root $r$ of
this interval is a maximal chain in the interval $[\hat 0, x]$. Given a maximal
chain $c_0$ in $C([x,y])$ we obtain (with respect to $r$) a labeling of edges of $c_0$ induced from
the labeling of a maximal chain $c' \in C[x,y]$ obtained by composing $r$,
$c_0$ and an arbitrary maximal chain in interval $[y, \hat 1]$. This labeling
does not depend on the choice of the chain in $[y,\hat 1]$ due to the
definition of chain-edge labeling. In the sequel, we consider the labeling of
$c_0$ as a sequence of $\rk(y) - \rk(x)$ elements of $\Lambda$. 
In particular, we can say that $c_0$ is increasing (in $[x,y]_r$) 
if its labeling is increasing and $c_0$ is lexicographically smaller than another maximal chain $c_1$ in $C([x,y])$ if the labeling of $c_0$ is lexicographically smaller then
the labeling of $c_1$.

We say that an chain-edge labeling is \emph{CL-labeling} (chain-lexicographic
labeling) if for every rooted interval $[x,y]_r$ in $P$ 
the following two conditions are satisfied.
\begin{enumerate}[(i)]
  \item There is a unique maximal increasing chain $c_0$ in $[x,y]_r$; and
  \item $c_0$ is lexicographically smaller than any other maximal chain in
    $[x,y]_r$.
\end{enumerate}
The poset $P$ is \emph{(chain-)lexicographically
shellable}, abbreviated as CL-shellable, if it admits CL-labeling.

It follows from~\cite{bjorner-wachs82} that every CL-shellable poset is
indeed shellable. Actually, the order of shelling is given by the lexicographic
order of chains in $C(P)$ (with respect to given CL-labeling).
The converse is not true---there are posets which are shellable but not
lexicographically shellable; see~\cite{vince-wachs85, walker85}.

\subsection{Lexicographic shellability versus $A$-shellability}

In this subsection we want to compare $A$-shellability and lexicographic
shellability. This comparison make sense if $A = \Aall$ is the set of all
atoms. In addition, we also assume that $\Aall$ is linearly ordered. (If we
allow arbitrary partial order on $\Aall$, then, for example, we can allow all
elements incomparable; then $\Aall$-shellability just coincides with
shellability.) 

\heading{Lexicographic shelling is an $\Aall$-shelling.}
Let $P$ be a CL-shellable poset and let us fix a $CL$-labeling
of it. Given an atom $a$ of $P$ we observe that the edge $e_a = \hat 0a$ is
labeled the same way in all maximal chains containing $e_a$ (by the definition
of chain-edge labeling). Thus, we can denote by $\lambda(e_a)$ this label of
$e_a$. By condition (ii) of the definition of CL-labeling we have that
$\lambda(e_a)$ and $\lambda(e_{a'})$ differ for two different atoms $a$ and
$a'$, and in addition they are comparable with $\Lambda$. Thus these labels
induce a linear ordering $\leq_\lambda$ on $\Aall$.
In this setting, the CL-shelling of $P$ is also an $\Aall$-shelling
of $P$ (where $\Aall$ is equipped with $\leq_\lambda$).

\heading{$\Aall$-shelling which is not lexicographic shelling.}
It is not hard to come up with an example of an $\Aall$-shelling which is not
a CL-shelling.
%
%
Let $P'$ be
a poset which is shellable but not CL-shellable. Let us consider $k$ copies $\hat 0_1, \dots,
\hat 0_k$ of the minimal element in $P'$. The poset $P$ is obtained by replacing
the minimal element of $P'$ by these $k$ copies and then adding a new minimal
element $\hat 0_{new}$ smaller than everything else. Note that $\Aall = \{\hat
0_1, \dots, \hat 0_k\}$.

It is not hard to check that $P$ is $\Aall$-shellable where $\Aall$ is equipped
with an arbitrary linear order (either by elementary means or by using
Theorem~\ref{t:ci}). On the other hand, $P$ is not CL-shellable
since $P$ contains an interval isomorphic to $P'$ and all intervals in a
CL-shellable poset are CL-shellable as well.

\subsection{Recursive atom orderings}
Bj\"{o}rner and Wachs~\cite{bjorner-wachs83} gave an equivalent reformulation
of CL-shellability using recursive atom orderings. It is useful to compare
$A$-shellability and recursive atom orderings. We first repeat their definition.

A poset $P$ (graded, with a unique minimum and maximum) 
admits a \emph{recursive atom ordering} if it has length 1
or if the length of $P$ is grater than 1 and there is an ordering $a_1, \dots,
a_t$ of all the atoms of $P$ which satisfies:

\begin{enumerate}[(R1)]
\item For all $k \in [t]$ the interval $[a_k,\hat 1]$ admits a recursive atom
  ordering in which the atoms of $[a_k,\hat 1]$ that come first in the ordering
  are those that cover some $a_i$ where $i < k$.
\item For all $i < k$, if $a_i, a_k < y$, then there is $j < k$ and an element
  $z$ such that $a_j,a_k \covered z \leq y$.
\end{enumerate}

Bj\"{o}rner and Wachs~\cite{bjorner-wachs83} proved that a poset is
CL-shellable if and only if it admits a recursive atom ordering.

In our notation a recursive atom ordering induces an ordering of $\Aall$.
From this point of view, recursive atom orderings are very strongly related 
to our second criterion, Theorem~\ref{t:cii}. Let us assume that condition~(i)
of Theorem~\ref{t:cii} is satisfied in slightly stronger form, that is, we
assume that $P\hull{A}$ admits a recursive atom ordering (which induces
an $A$-shelling). Similarly, let us assume that we can replace
$\Aall(a^+)$-shellability of $I(a^+)$ with a recursive atom ordering on
$I(a^+)$ inducing $\Aall(a^+)$-shellability. Then we can deduce that
$P\hull{A^+}$ admits a recursive atom ordering:

Indeed condition~(R1) translates to condition~(ii) of Theorem~\ref{t:cii} (it
is sufficient to check (R1) only for $a_k = a^+$ since we already assume that
$P\hull{A}$ admits a recursive atom ordering). Similarly we will check that 
condition~(R2) translates to condition~(iii) of Theorem~\ref{t:cii}. 
Given $a_i$, $a_k$ and $y$ from (R2) we can again assume that $a_k = a^+$. We
choose a maximal chain $c$ in $[a^+,y]$ and set $p$ to be the smallest element
of $c$ belonging to $P\hull{A}$ (note that $y \in P\hull{A}$ since $a_i < y$);
see Figure~\ref{f:rao}.
Then we can set $q$ to be the element of $c$ one rank below $p$. Then,
by assuming~(iii) of Theorem~\ref{t:cii}, $p$ is above some $z \in A(a^+)$.
This is the required $z$ since $z \in A(a^+)$ implies that $z$ covers some atom
$a_j$ preceding $a_k$.

\begin{figure}
\begin{center}
  \includegraphics{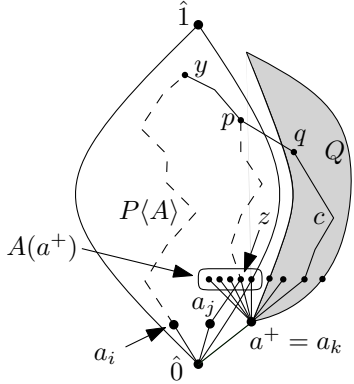}
  \caption{Condition~(iii) of Theorem~\ref{t:cii} follows from (R2). We also
    recommend to compare
  this picture with Figure~\ref{f:cii}.}
  \label{f:rao}
\end{center}
\end{figure}

Altogether, we see that the method using $A$-shellability includes the
recursive atom ordering method. On the other hand, it is not hard to see, that
if we were allowed to use only Theorem~\ref{t:cii}, we would not get more than
recursive atom orderings. However, Theorem~\ref{t:cii} is still more flexible
since, for example, it does not need to assume that $P\hull{A}$ comes with a
recursive atom ordering. This is useful, when it is combined with
Theorem~\ref{t:ci}.

\subsection{Lexicographic shellability versus Theorem~\ref{t:ci}.}
Now we compare our first criterion, Theorem~\ref{t:ci}, to lexicographic
shellability (in this case it is more natural to choose lexicographic
shellability rather than recursive atom orderings). In this case,
Theorem~\ref{t:ci} seems to be in more `generic' position in relation with
lexicographic shellability.

\heading{CL-shellable poset which does not satisfy assumptions
of Theorem~\ref{t:ci}.} First we provide an example of a poset that is
CL-shellable, but which does not satisfy assumptions of
Theorem~2, with respect to a given CL-shelling. This example 
arose in discussions with Afshin Goodarzi.

Let $P$ be the poset from Figure~\ref{f:lex}. It is lexicographically shellable:
we first label edges as on picture; and then we label chains according to
labels of edges. The reader is welcome to check that we indeed obtain a
CL-labeling. (Actually, we obtain a so called \emph{EL-labeling} where, in
addition, the label of an edge does not depend on the considered chain.) 
Note also, that chains containing $a$ appear before chains containing $a^+$ 
in the corresponding lexicographic shelling. In particular, $P$ is
$A^+$-shellable where $A^+ := \{a, a^+\}$ and $a$ appears before $a^+$.

\begin{figure}
\begin{center}
  \includegraphics{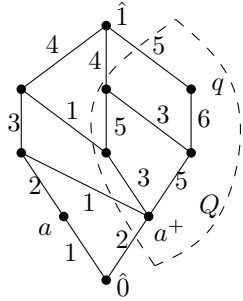}
\end{center}
\caption{Lexicographically shellable poset which does not satisfy assumptions
of Theorem~\ref{t:ci}.}
\label{f:lex}
\end{figure}

On the other hand, if we intend to use Theorem~\ref{t:ci} for showing
$A^+$-shellability of $P$, we will not succeed. The condition (iii) (edge
falling property) is not satisfied for the edge $q\hat1$. 

\heading{Theorem~\ref{t:ci} does not provide a CL-shelling.}
Let us imagine that we replace our shellability assumptions in
Theorem~\ref{t:ci} by CL-assumptions. That is, for
condition~(i) we would assume that $P\hull{A}$ is CL-shellable
(and the corresponding CL-shelling is $A$-shelling as well); and for
condition~(iv) we would assume that $I(q)\hull{A(q)}$ is CL-shellable. Does it
follow that $P\hull{A^+}$ is CL-shellable?

The author does not know the answer to this question; but it seems that the more
probable answer is `no'. If the answer is indeed `no', then this would mean
further difference in applicability of Theorem~\ref{t:ci} and CL-shellability
(or even more general CC-shellability of Kozlov~\cite{kozlov97} as remarked
below). However, even if the answer is `yes', Theorem~\ref{t:ci} still
provides particular conditions that might possibly be checked in an easier way
than establishing CL-labeling (or establishing recursive atom ordering).

Here, we can at least provide a simple example showing that the current proof of
Theorem~\ref{t:ci} does not provide CL-shelling even if we ask CL-shelling assumptions. Let $P$ be the poset on
Figure~\ref{f:non_lex_order}. If we set $a^+$ as in the picture, we can easily
check that all assumptions of Theorem~\ref{t:ci} are satisfied even with
lexicographic assumptions. We label elements of $Q$ as $q_1, \dots, q_5$
consistently with the proof of Theorem~\ref{t:ci}. Then the proof provides
shelling such that the chains $\hat0q_1q_2q_4\hat1$, $\hat0q_1q_3q_4\hat1$,
$\hat0q_1q_2q_5\hat1$, and $\hat0q_1q_3q_5\hat1$ appear in this order; consult also
Figure~\ref{f:i_order}. This cannot be a CL-shelling due to the
alternation of edges $q_1q_2$ and $q_1q_3$. (The reader familiar with Kozlov's
CC-shellability~\cite{kozlov97} is welcome to check that this is not even a
CC-shelling.) 
\begin{figure}
\begin{center}
  \includegraphics{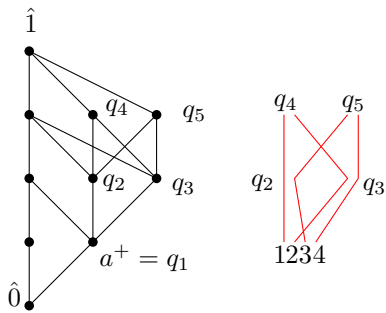}
\end{center}
\caption{Theorem~\ref{t:ci} does not produce a lexicographic shelling of
this poset.}
\label{f:non_lex_order}
\end{figure}

\section*{Acknowledgment}
I am very thankful to Anders Bj\"{o}rner, Aldo Conca and Afshin Goodarzi for
valuable discussions. In addition, I further thank to Anders Bj\"{o}rner
for suggesting this topic for research and reading preliminary versions of this
paper, Aldo Conca for letting me know that
Corollary~\ref{c:koszul} follows from results in~\cite{concaherzogtrungvalla97}
and Volkmar Welker for
pointing out the reference~\cite{peeva-reiner-sturmfels98}.
Finally, I thank two anonymous referees for many valuable remarks.
\bibliographystyle{alpha}
\bibliography{/home/martin/clanky/bib/general}

\end{document}